\documentclass[10pt,a4paper]{amsart}
\usepackage[english]{babel}
\usepackage{mathtools}
\usepackage{amsmath}
\usepackage{amssymb}
\usepackage{amsfonts}
\usepackage{amsthm}
\usepackage{amsxtra}
\usepackage{mathtools}
\usepackage{portland}
\usepackage{rotating}
\usepackage{nicefrac}
\usepackage{float}
\usepackage[all,web,arc,poly,dvips,color]{xy}
\usepackage{epic,eepic}
\usepackage{epsfig}
\usepackage[dvips]{color}
\usepackage{array}
\usepackage{bm}
\usepackage[sc,osf]{mathpazo}
\usepackage{pifont}
\usepackage{multirow}
\usepackage{latexsym}
\usepackage{graphics}
\usepackage{fullpage}
\usepackage{hyperref}
\usepackage{endnotes}
\usepackage{footmisc}
\usepackage[T1]{fontenc}
\usepackage[utf8]{inputenc}

\newcommand{\mathsym}[1]{{}}
\newcommand{\unicode}[1]{{}}

\theoremstyle{plain}
\newtheorem{corollary}{Corollary}
\newtheorem{lemma}{Lemma}
\newtheorem{proposition}{Proposition}

\theoremstyle{definition}
\newtheorem{definition}{Definition}

\theoremstyle{remark}
\newtheorem{remark}{Remark}

\newcommand{\C}{\mathbb C}
\newcommand{\R}{\mathbb R}
\newcommand{\Z}{\mathbb Z}

\newcommand{\DySt}{\displaystyle}

\newcommand{\half}{
        {\lower0.00ex\hbox{\raise.6ex\hbox{\the\scriptfont0 1}
                           \kern-.5em\slash\kern-.1em\lower.45ex
                                     \hbox{\the\scriptfont0 2}}}}
\newcommand{\quarter}{
        {\lower0.00ex\hbox{\raise.6ex\hbox{\the\scriptfont0 1}
                           \kern-.5em\slash\kern-.1em\lower.45ex
                                     \hbox{\the\scriptfont0 4}}}}
\newcommand{\tquarter}{
        {\lower0.00ex\hbox{\raise.6ex\hbox{\the\scriptfont0 3}
                           \kern-.5em\slash\kern-.1em\lower.45ex
                                     \hbox{\the\scriptfont0 4}}}}
\newcommand{\eighth}{
        {\lower0.00ex\hbox{\raise.6ex\hbox{\the\scriptfont0 1}
                           \kern-.5em\slash\kern-.1em\lower.45ex
                                     \hbox{\the\scriptfont0 8}}}}
\newcommand{\othird}{
        {\lower0.00ex\hbox{\raise.6ex\hbox{\the\scriptfont0 1}
                           \kern-.5em\slash\kern-.1em\lower.45ex
                                     \hbox{\the\scriptfont0 3}}}}
                                     
\def\cydot{\leavevmode\raise.4ex\hbox{.}}

\begin{document}

\title[]
{The diagonal two-point correlations of the Ising model on the anisotropic triangular lattice and Garnier systems}

\author{N.S.~Witte}
\address{School of Mathematics and Statistics,
University of Melbourne,Victoria 3010, Australia}
\email{\tt nsw@ms.unimelb.edu.au}

\begin{abstract}
The diagonal spin-spin correlations $ \langle \sigma_{0,0}\sigma_{N,N} \rangle $ of the Ising model on a triangular
lattice with general couplings in the three
directions are evaluated in terms of a solution to a three-variable extension of the sixth Painlev\'e system, namely a Garnier
system. This identification, which is accomplished using the theory of bi-orthogonal polynomials on the unit circle with regular 
semi-classical weights, has an additional consequence whereby the correlations are characterised by a simple 
system of coupled, nonlinear recurrence relations in the spin separation $ N \in \Z_{\geq 0} $. These later recurrence relations are
an example of the discrete Garnier equations which, in turn, are extensions to the ``discrete Painlev\'e V'' 
system.
\end{abstract}

\maketitle

In this study we investigate the spin-half Ising model on the anisotropic, homogeneous triangular lattice
and in particular the diagonal spin-spin correlations of the model.
We identify a particular three-variable Garnier system in Corollary \ref{GarnierID} as 
the integrable system lying behind this model and thus find an extension of the evaluation by
Jimbo and Miwa in 1980 \cite{JM_1980} of the same quantity for the square lattice case in terms of a $ \tau $ function
of the sixth Painlev\'e system. This extension is quite literal in the sense that the anisotropic triangular
model includes the diagonal and row or column correlation functions of the rectangular lattice as special cases.
To our knowledge this is the first appearance of a Garnier system in a statistical mechanical model even
though such systems have arisen in other mathematical physics contexts: see \cite{MS_2000} for the ADM approach
to $2\!+\!1$-dimensional gravity; \cite{IP_1995} for the symmetry reductions of the self-dual Yang-Mills
equations in $2\!+\!2$ dimensions; \cite{Ton_1994} for a generalised H\'enon-Heiles system; and \cite{Kos_1989} for
finite gap solutions to the Garnier system and the $g$-dimensional anisotropic harmonic oscillator in a 
radial quartic potential.

The first identification of a Painlev\'e system for the Ising model correlations was for the scaled and critical correlations\footnote{
Where the limits $ N \to \infty $ and $ T \to T_C $ are carried out together in a certain combination.} for the square lattice
and in this case was made to a special type of the $ D^{(1)}_6 $ Painlev\'e III system by Wu, McCoy, Tracy and Barouch in \cite{WMTB_1976}.
The extension to the lattice model for all $ T $ and finite $ N $ was found by Jimbo and Miwa \cite{JM_1980}.
This latter result for the spin-spin correlation $ \langle \sigma_{0,0}\sigma_{N,N} \rangle $ can be summarised as
\begin{equation}
   \sigma_{N}(t) = \left\{
   \begin{array}{lr}
     t(t-1)\frac{\displaystyle d}{\displaystyle dt}\log\langle \sigma_{0,0}\sigma_{N,N} \rangle-\frac{1}{4},  & T<T_C \\
     &\\
     t(t-1)\frac{\displaystyle d}{\displaystyle dt}\log\langle \sigma_{0,0}\sigma_{N,N} \rangle-\frac{1}{4}t, & T>T_C
   \end{array}     \right. ,
\end{equation}
where $ \sigma_{N}(t) $ is the solution to the $ \sigma$-form of Painlev\'e VI which is the second order, second degree ordinary
differential equation
\begin{equation}
   \left[ t(t-1)\frac{d^2 \sigma_{N}}{dt^2} \right]^2 
     = N^2\left[ (t-1)\frac{d\sigma_{N}}{dt}-\sigma_{N} \right]^2
     -4\frac{d\sigma_{N}}{dt}\left[ (t-1)\frac{d\sigma_{N}}{dt}-\sigma_{N}-\tfrac{1}{4} \right]\left[ t\frac{d\sigma_{N}}{dt}-\sigma_{N} \right] .
\label{JM1980}
\end{equation}
Here  $ t=k^{-2} $ and $ k=\sinh 2K_1 \sinh2K_2 $ where $ K_1, K_2 $ are the couplings between sites along the horizontal
and vertical edges respectively. This equation is solved subject to the boundary condition as $ t \to 0 $, i.e. in the $ T<T_C $ regime
\begin{equation}
  \langle \sigma_{0,0}\sigma_{N,N} \rangle(t) = (1-t)^{1/4} + \frac{(1/2)_{N}(3/2)_{N}}{4[(N+1)!]^2}t^{N+1}(1+{\rm O}(t)) .
\end{equation}
Alternative derivations of this result can be found in \cite{BD_2002} and \cite{FW_2004a}.

Equation \eqref{JM1980} is a specialisation of the general $ \sigma$-form of Painlev\'e VI which is constructed from
the Hamiltonian system for Painlev\'e VI \cite{Mal_1922}, \cite{Oka_1987}
\begin{equation}
   \frac{dq}{dt} = \frac{\partial H}{\partial p}, \qquad \frac{dp}{dt} = -\frac{\partial H}{\partial q} ,
\end{equation}
where the non-autonomous Hamiltonian is given by
\begin{equation}
   t(t-1)H = q(q-1)(q-t)p^2  \\ - [\alpha_4(q-1)(q-t)+\alpha_3 q(q-t)+\alpha_0 q(q-1)]p + \alpha_2(\alpha_1+\alpha_2)(q-t) .
\end{equation}
Here the $ D_4 $ root system constraint $ \alpha_0+\alpha_1+2\alpha_2+\alpha_3+\alpha_4 = 1 $ applies.
The general $\sigma$-form of the sixth Painlev\'e equation is then
\begin{equation} 
   h'(t)\left[ t(1-t)h''(t) \right]^2 + \left[ h'(t)(2h(t)-(2t-1)h'(t))+b_1b_2b_3b_4 \right]^2 = \prod^{4}_{k=1}(h'(t)+b_k^2) ,
\end{equation} 
and $ h(t) $ is related to the Hamiltonian by
\begin{equation}
   h(t) = t(t-1)H(t)+e_2(b_1,b_3,b_4)t-\tfrac{1}{2}e_2(b_1,b_2,b_3,b_4) ,
\end{equation}
where $ e_2(\cdot) $ is the second elementary symmetric polynomial of its arguments.

In the present work we require the theory for the multi-variate extension of the Painlev\'e VI system.
Let us place $ M+1 \geq 3 $ independent variables in canonical position (see \cite{IKSY_1991}), i.e. taking them at distinct points
\begin{equation}
   t_0 = 0, t_1, \ldots, t_{M-2}, t_{M-1}=1, t_{M}=\infty, 
\end{equation}
each of which is associated with the parameters $ \rho_0, \rho_1, \ldots, \rho_{M-2}, \rho_{M-1}=\rho, \rho_{M}=\rho_{\infty} $ 
respectively.
The Garnier system $ {\mathcal G}_{M-2} $ of type $ L(1^{M};M-2) $ is labelled by an abbreviated Riemann-Papperitz-like symbol
(see \S 4 of \cite{IKSY_1991} for the definition) in the following manner
\begin{equation}
	\left\{
\begin{array}{cccccc}
	t_0=0    	  & t_1     	     & \cdots & t_{M-2}     	 & t_{M-1}=1 	    & \infty                     \\
	\theta_0=n-\rho_0 & \theta_1=-\rho_1 & \cdots & \theta_{M-2}=-\rho_{M-2} & \theta=-\rho     & \theta_{\infty}=2n+1-\sum^{M-1}_{j=0}\rho_j  
\end{array}
	\right\} .
\end{equation}
The dynamics of the Garnier system is governed by the 
Hamiltonian system $ \{q_j,p_j;K_j,t_j\}^{M-2}_{j=1} $ with co-ordinate $ q_r $ and momenta $ p_r $ and with the Hamiltonian \cite{KO_1984,IKSY_1991,Wit_2009}
\begin{equation}
  K_j = \frac{\Theta_n(t_j)}{W'(t_j)}\sum^{M-2}_{r=1}\frac{W(q_r)}{\Theta'_n(q_r)}\frac{1}{t_j-q_r}
    \left[ p^2_r+p_r\left( \frac{2V(q_r)}{W(q_r)}-\frac{n}{q_r}-\frac{1}{t_j-q_r} \right)
      -\frac{n(1+m_0)}{q_r(q_r-1)} 
    \right] .
\label{GarnierHam}
\end{equation}
The monodromy exponents are given by $ \theta_j = -\rho_j $ for $ j=1,\ldots, M-1 $,
$ \theta_0 = n-\rho_0 $, and $ \theta_{\infty} = 2n+1-\sum^{M-1}_{j=0}\rho_j $ 
with the constant $ \kappa = -n(1+m_0) $.
The polynomial $ W $, the denominator spectral polynomial, is given by 
\begin{equation}
   W(z) = z(z-1)\prod^{M-2}_{j=1}(z-t_j)
        = z\sum^{M-1}_{l=0} (-)^{M-1-l}z^l e_{M-1-l} ,
\label{Wdefn}
\end{equation}
where the elementary symmetric functions of the singularity positions are denoted 
$ e_l, l=0, \ldots,M-1 $ and in particular $ e_0=1 $, $ e_{M}=0 $ and $ e_{M-1}=\prod^{M-2}_{j=1}t_j $.
The polynomial $ 2V $, the numerator spectral polynomial, is 
\begin{equation}
   2V(z) = z(z-1)\prod^{M-2}_{j=1}(z-t_j)
             \left\{ \frac{\rho_0}{z}+\frac{\rho}{z-1}+\sum^{M-2}_{j=1}\frac{\rho_j}{z-t_j} \right\}
         = \sum^{M-1}_{l=0} (-)^{M-1-l}z^l m_{M-1-l} ,
\label{2Vdefn}
\end{equation}
where the last relation defines the coefficients $ m_l, l=0,\ldots,M-1 $ and 
we observe that $ m_0=\rho_0+\rho+\sum^{M-2}_{j=1}\rho_j $ and $ m_{M-1}=\rho_0 e_{M-1} $.
The remaining polynomial, termed the spectral coefficient $ \Theta_n $, of degree $ M-2 $, is 
\begin{equation}
 \Theta_n(z) = \Theta_{\infty}\prod^{M-2}_{r=1}(z-q_r) , \quad
 \Theta_{\infty} = (n+1+m_0)\frac{\kappa_n}{\kappa_{n+1}} .
\label{ThetaRep}
\end{equation}

This Hamiltonian system \eqref{GarnierHam} is not polynomial in the canonical variables $ q_r $ and the dynamical 
equations for $ q_r $ do not possess the Painlev\'e property in $ t_j $, however using the well-known 
canonical transformation to the new Hamiltonian system $ {\mathcal H}_{M-2} = \{Q_j,P_j;H_j,s_j\} $ 
\cite{KO_1984}, \cite{IKSY_1991}
\begin{align}
   s_j & = \frac{t_j}{t_j-1} ,
\label{Hfix:a} \\
   Q_j & = \frac{t_j\Theta_n(t_j)}{\Theta_{\infty}W'(t_j)} ,
\label{Hfix:b} \\
   P_j & = -(t_j-1)\sum^{M-2}_{r=1}\frac{p_r}{q_r(q_r-1)}
    \frac{\Theta_{\infty}W(q_r)}{(q_r-t_j)\Theta'_n(q_r)} ,
\label{Hfix:c}
\end{align}
both these deficiencies can be removed.
The sixth Painlev\'e system is the $ M=3 $ case of the above.

The $\sigma$-form differential equation \eqref{JM1980} is quite useful in certain types of analysis of the correlations
however additional insight can be gained from a recurrence system in $ N $ which generates the correlation at $ N+1 $ from ones at $ N $
or earlier, see for example the use made of these in \cite{CGNP_2011}.
In Propositions \ref{SLdiagonalCorr} and \ref{recoverSLdiagonalCorr} a system of two coupled, first order non-linear difference equations for the diagonal correlations of 
the square lattice Ising model is given as ($ \alpha = k^{-1} $)
\begin{gather}
   \alpha^{-2}f_{n}f_{n+1} = \frac{\left[ g_{n}+(n+1)\alpha-\alpha^{-1} \right]\left[ g_{n}+(n+\frac{1}{2})\alpha-\frac{1}{2}\alpha^{-1} \right]}
                                  {\left[ g_{n}+n\alpha^{-1} \right]\left[ g_{n}+(n+\frac{1}{2})\alpha^{-1}-\frac{1}{2}\alpha \right]} ,
\\
   g_{n}+g_{n-1}+2n\alpha^{-1}-\tfrac{1}{2}(\alpha+\alpha^{-1}) + (n+\tfrac{1}{2})\frac{(\alpha^2-1)}{\alpha}\frac{1}{1-f_{n}} + (n+\tfrac{1}{2})\frac{\alpha(\alpha^2-1)}{\alpha^2-f_{n}} = 0 ,
\end{gather}
although different but entirely equivalent systems have been given in \cite{FW_2005} and \cite{Wit_2007}.
The above system is a specialisation of the general ``discrete Painlev\'e V'' system \cite{GORS_1998}, \cite{NRGO_2001}
or the $ D^{(1)}_4 $ system in Sakai's classification \cite{Sak_2001}
\begin{equation}                  
  tf_nf_{n+1}
  = \frac{[\omega_n+n-t-\rho_0(t+1)-(\rho_t+\rho_1)t][\omega_n+n-t-\rho_0(t+1)-\rho_t-\rho_1t]}
           {[\omega_n+nt-1-\rho_0(t+1)-\rho_t-\rho_1][\omega_n+nt-1-\rho_0(t+1)-\rho_t-\rho_1t]} ,
\end{equation}
and
\begin{equation}       
  \omega_n+\omega_{n-1}+(2n-1)t-2-2\rho_0(t+1)-2\rho_t-\rho_1(t+1) 
  = (n-\rho_0)\frac{1-t}{f_n-1}+(n+1+m_0)\frac{1-t}{tf_n-1} .
\end{equation}
All of the features that apply in the diagonal correlations of the square lattice case as a 
consequence of the identification with an integrable system, also apply in the case of the triangular lattice
but only in a slightly more complicated way. Thus the main results of our study are the systems of non-linear
recurrence relations for the diagonal correlations of the triangular Ising model in Corollary \ref{NonlinearRR} and
\ref{Recover}, and by implication for the column/row correlations of the anisotropic square lattice Ising model
in Corollary \ref{NonlinearRR_ColCorrSqL} and Corollary \ref{Recover_ColCorrSqL}. We will not undertake the
task of writing down nor analysing here the coupled, partial differential equations that follow from 
\eqref{GarnierHam} and which directly generalise \eqref{JM1980} as this would divert us from our primary goal.

The Ising model on the triangular lattice has attracted some considerable interest, partly because it exhibits the phenomena of 
frustration in the anti-ferromagnetic phase and partly because it led to the simplest form for the diagonal correlations on the 
square lattice. The partition function has been evaluated through all the methods known to work in treating the Ising model: 
Husimi and Sy{\^o}zi \cite{HS_1950}, \cite{Syo_1950} used a simplified version of Onsager's algebraic method on the isotropic
lattice; Wannier \cite{Wan_1950} used the method of Kaufman and Onsager for the isotropic lattice whereas Temperley \cite{Tem_1950}
employed this approach on the anisotropic lattice; Newell first employed Kaufman's method in \cite{New_1950a} and then Kaufman
and Onsager's method on the anisotropic lattice in \cite{New_1950}; Potts employed the combinatorial method of Kac and Ward on
the anisotropic model in \cite{Pot_1955}. In all of these works the isotropic lattice was observed to possess a Curie point
singularity in the ferromagnetic case, just like that for the square lattice, but no singularity was found in the antiferromagnetic
case. Subsequently Stephenson produced a series of works I-IV, \cite{Ste_1964, Ste_1966, Ste_1970a, Ste_1970b}, starting with the
Pfaffian representation of the partition function introduced by Kastelyn and utilised by Montroll, Potts and Ward to compute the
square lattice row/column correlations \cite{MPW_1963}. It is Stephenson's first and fourth works that we will primarily draw
upon here. In his first work the two-spin correlations are evaluated as a Toeplitz determinant; in the second some particular
four-spin correlations are evaluated; in the third the isotropic anti-ferromagnetic lattice is studied using the two and four-spin
correlations and the asymptotics found for these at large site separations; and in the fourth we find the most detailed
treatment of the anisotropic lattices in both the ferromagnetic and anti-ferromagnetic phases. Subsequent to these earlier studies
a number of modern re-derivations of the partition function (and extensions thereof) \cite{Ple_1988}, \cite{Bug_1996} have been 
made by exploiting the free-fermionic character of the model using Grassmann variable techniques. 

The triangular Ising model is known to be in the Villain-Stephenson universality class with critical exponents 
$ \alpha=0, (2-\alpha)/\nu=2, \nu=1, \beta/\nu=1/8, \delta=7, \gamma/\nu=7/4, \Delta/\nu=7/4 $, see for example \cite{HTM_1992}. 
As mentioned above because the triangular Ising model in the anti-ferromagnetic regime is a prototype for geometrical
frustration an extensive modern literature has emerged which we will only briefly touch upon. For example there are the studies
\cite{ZS_2005}, \cite{LYC_2008}, \cite{YLCM_2008} of the Ising model on a triangular Kagom\'e lattice which explore this issue.
Another feature of this effect is that the ground state of the model is highly degenerate, $ \# configurations = {\rm O}(e^{\# sites}) $,
and has residual entropy at $ T=0 $. A consequence of this is that the boundary conditions do affect this entropy as
$ \# sites \to \infty $ and some computations \cite{MGP_2003}, \cite{BM_2006} have clearly illustrated this.
Perhaps more relevant to this work is the study \cite{WM_2009} of the pair (both diagonal and off-diagonal) correlations of the
isotropic triangular Ising model in the anti-ferromagnetic phase at $ T=0 $ and $ T>0 $ through a number of approaches:
the exact results for $ N\leq 20 $ numerically evaluated, asymptotic approximations as $ N \to \infty $ and Monte-Carlo simulations.
Here it should be noted that in the isotropic anti-ferromagnetic
case the model has no long-range order for $ T>0 $, i.e. the critical point is at $ T=0 $. What is observed here is the
lack of exact results even in this highly specialised case and for $ T>0 $ the existing approximations are of limited
accuracy.

\subsection{Correlations along the diagonal of the anisotropic triangular Ising model}

The Ising model is a system of spins $ \sigma_r \in \{-1,1\} $ located at site $ r=(i,j) $ on a
triangular lattice of dimension $(2L+1) \times (2L+1)$, or equivalently with the sites on a 
rectangular lattice having in additional to nearest-neighbour couplings along the $x$ and $y$ 
axes a third next-nearest-neighbour coupling along the up-right diagonals. Our conventions for
the labelling of sites and of the coupling constants $ K_{i}, i=1,2,3 $ are displayed in
Fig. \ref{TRaxes}, with the origin at the centre of the lattice.

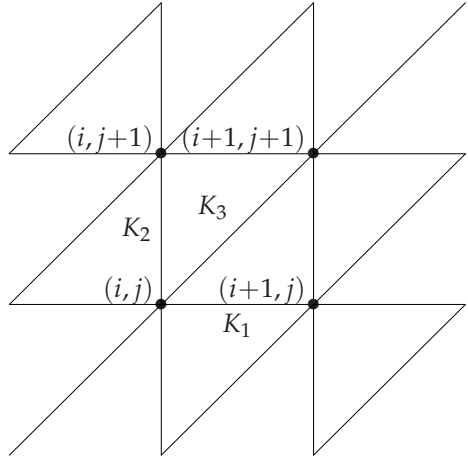
\begin{figure}[H]
\[
 \begin{xy}
    *\xybox{0;<2cm,0cm>:<0cm,2cm>::
           ,0="O"*{\bullet}*!RD{(i,j)\;}
           ,{\ar@{-}_{\DySt K_1} "O";"O"+<2cm,0cm>}
           ,{\ar@{-}^{\DySt K_2} "O";"O"+<0cm,2cm>}
           ,{\ar@{-}^{\DySt K_3} "O";"O"+<2cm,2cm>}
     ,"O";"O"+<-2cm,0cm>**@{-}
     ,"O";"O"+<0cm,-2cm>**@{-}
     ,"O";"O"+<-2cm,-2cm>**@{-}
     ,"O"+(0,1)="U"*{\bullet}*!RD{(i,j\!+\!1)\;}
     ,"U";"U"+<2cm,0cm>**@{-}
     ,"U";"U"+<-2cm,0cm>**@{-}
     ,"U";"U"+<-2cm,-2cm>**@{-}
     ,"U";"U"+<2cm,2cm>**@{-}
     ,"U";"U"+<0cm,2cm>**@{-}
     ,"U"+(1,0)="D"*{\bullet}*!RD{(i\!+\!1,j\!+\!1)\;}
     ,"D";"D"+<2cm,0cm>**@{-}
     ,"D";"D"+<0cm,-2cm>**@{-}
     ,"D";"D"+<0cm,2cm>**@{-}
     ,"D";"D"+<2cm,2cm>**@{-}
     ,"O"+(1,0)="R"*{\bullet}*!RD{(i\!+\!1,j)\;}
     ,"R";"R"+<2cm,0cm>**@{-}
     ,"R";"R"+<0cm,-2cm>**@{-}
     ,"R";"R"+<-2cm,-2cm>**@{-}
     ,"R";"R"+<2cm,2cm>**@{-}
     ,"R"+(0,-1)="RD"
     ,"RD";"RD"+<2cm,2cm>**@{-}
     ,"U"+(-1,0)="UL"
     ,"UL";"UL"+<2cm,2cm>**@{-}
           }
 \end{xy}
\]
\caption{Conventions for the co-ordinate system of the spin sites $ (i,j) $ and couplings $ K_1, K_2, K_3 $ for the homogeneous,
anisotropic triangular lattice Ising model.}
\label{TRaxes}
\end{figure}

The probability density function for a configuration $ \{\sigma_{i,j}\}_{i,j=-L}^{L} $ is the product of Boltzmann
weights along the three axes  
\begin{multline}
 {\mathbb P}[\{\sigma_{i,j}\}_{i,j=-L}^{L}] = \frac{1}{Z_{2L+1}}
     \exp \Big[ K_1 \sum_{j=-L}^L \sum_{i=-L}^{L-1}\sigma_{i,j}\sigma_{i+1, j}
               + K_2 \sum_{i=-L}^L \sum_{j=-L}^{L-1}\sigma_{i,j}\sigma_{i, j+1} \\
                + K_3 \sum_{i=-L}^{L-1} \sum_{j=-L}^{L-1}\sigma_{i,j}\sigma_{i+1, j+1} \Big] ,
\end{multline}
where $ Z_{2L+1} $ is the partition function.
Averages such as the order parameter or average magnetisation are defined by
\begin{equation}
  \langle \sigma_{0,0} \rangle
  = \sum_{\{\sigma_{i,j}\}\in (-1,1)^{2L+1}} \sigma_{0,0}{\mathbb P}[\{\sigma_{i,j}\}_{i,j=-L}^{L}] ,
\end{equation}
however our interest will lie in the spin-spin correlation function
\begin{equation}
  \langle \sigma_{0,0}\sigma_{N,N} \rangle
  = \sum_{\{\sigma_{i,j}\}\in (-1,1)^{2L+1}} \sigma_{0,0}\sigma_{N,N}{\mathbb P}[\{\sigma_{i,j}\}_{i,j=-L}^{L}] .
\end{equation} 
In all the averages we will take the thermodynamic limit, so for example  
$ \lim_{L\to \infty}\langle \sigma_{0,0} \rangle=\langle \sigma \rangle $. 
Our independent variables will either be the ordered triple $ (z_1,z_2,z_3) $ or $ (v_1,v_2,v_3) $
which are related to the coupling constants $ K_{i}, i=1,2,3 $ by
\begin{equation}
  z_i = \tanh K_i = \frac{1-v_i}{1+v_i}, \quad v_i=e^{-2K_i} = \frac{1-z_j}{1+z_j}, 
\end{equation}
where the physical variables $ z_i \in [-1,1] $, $ v_i \in [0,\infty) $ for $ i=1,2,3 $\footnote{
This notation is the reverse of that adopted by Stephenson whereas our conventions conform to common usage, see \cite{MW_1973}.}.
In our characterisation of these correlations as a classical solution to the Garnier equations we will
see that they are meaningful for complex values of the variables 
$ K_1, K_2, K_3 \in \C\cup\{\infty\} $ modulo certain restrictions and in fact this is the natural setting for these systems.

From the symmetries of the correlations, see Eq. (1.6) of \cite{Ste_1970b}, we have the Toeplitz matrix element
\begin{align}
   w_n(K_1,K_2,K_3)
   & = w_n(-K_1,-K_2,K_3) ,\\
   & = (-1)^{n+1}w_n(-K_1,K_2,-K_3) ,\\
   & = (-1)^{n+1}w_n(K_1,-K_2,-K_3) ,
\end{align}
from which Stephenson inferred a classification with two classes: Class A which can be transformed to a completely
ferromagnetic lattice $ K_1, K_2, K_3 > 0 $ and class B which can be transformed to a completely
antiferromagnetic lattice $ K_1, K_2, K_3 < 0 $. 
In the class A system or ferromagnetic system there is a critical point, the {\it Curie point} $ T_C $, given by
\begin{equation}
  1+z_1 z_2 z_3 = z_1+z_2+z_3+z_1 z_2+z_1 z_3+z_2 z_3 ,\quad v_1v_2+v_1v_3+v_2v_3 = 1 ,
\label{CuriePt}
\end{equation} 
and exhibits a phase transition with a low temperature, $ 0 < T < T_C $ , ordered phase where the spins are aligned
$ \langle \sigma \rangle > 0 $ and a high temperature, $ T_C < T < \infty $, disordered phase with no alignment 
$ \langle \sigma \rangle = 0 $. By contrast, in the antiferromagnetic class or class B systems there are two critical points,
the {\it N\'eel point} $ T_N $
\begin{equation}
  1+z_1 z_2 z_3 = -z_1-z_2+z_3+z_1 z_2-z_1 z_3-z_2 z_3 ,\quad v_1v_2-v_1v_3-v_2v_3 = 1 ,
\label{NeelPt}
\end{equation} 
and the {\it disorder point} $ T_D > T_N $ 
\begin{equation}
  z_3+z_1 z_2 = 0, \quad -v_1v_2+v_1v_3+v_2v_3 = 1 .
\label{DisorderPt}
\end{equation} 
In the low-temperature regime $ 0 < T < T_N $ the system exhibits antiferromagnetic long-range-order along the two lattice
axes with the strongest $ K $ values and ferromagnetic order along the third. There is now an intermediate regime
$ T_N < T < T_D $ in which there is antiferromagnetic short-range-order along the two axes with the largest absolute 
values of $ K $ and ferromagnetic short-range-order along the third. In the high temperature regime $ T_D < T < \infty $
there is exponential decay of the pair correlations along all the axes. There is a clear exposition of the 3-D phase diagram
in \cite{Bug_1996} with respect to the co-ordinates $ (K_1, K_2, K_3) \in \R^3 $, where it is clear there exist rays
emanating from the origin with no critical point for $ T>0 $, i.e. the system is disordered down to $ T=0 $ such as 
in the isotropic anti-ferromagnetic case $ K_1=K_2=K_3< 0 $.

The spin-spin correlations along the diagonal have the Toeplitz determinant form given by Eqs. (6.10) and (6.12)
derived by Stephenson \cite{Ste_1964}, or Eq. (1.4) of \cite{Ste_1970b},
using the methods of Montroll, Potts and Ward \cite{MPW_1963}
\begin{equation}
  \langle \sigma_{0,0}\sigma_{N,N} \rangle^{\triangle} = \det[w_{j-k}]_{j,k=0,\dots,N-1} ,
\label{ssDiag}
\end{equation}
where the Toeplitz matrix element is the trigonometric integral
\begin{equation}
   w_{n} = \int^{\pi}_{-\pi}\frac{d\theta}{2\pi}\frac{a \cos(n\theta)-b \cos((n-1)\theta)-c \cos((n+1)\theta)}{\sqrt{a^2+b^2+c^2-2 a (b+c) \cos(\theta )+2 b c \cos(2 \theta )}}, \quad n \in \Z ,
\label{Telement}
\end{equation} 
with parameters $ a,b,c $ depending in the three couplings $ z_1, z_2, z_3 $
\begin{equation}
  a = 2z_3(1+z^2_1)(1+z^2_2)+4z_1z_2(1+z^2_3), \quad b = z^2_3c = z^2_3(1-z^2_1)(1-z^2_2) .
\end{equation}
For our purposes we note the essential feature that these matrix elements are the Fourier coefficients of the weight
\begin{equation}
  w(\zeta) \equiv \sum^{\infty}_{n=-\infty}w_{n}\zeta^n = \left(\frac{a-b\zeta-c\zeta^{-1}}{a-b\zeta^{-1}-c\zeta}\right)^{1/2} .
\label{ssSymbol}
\end{equation}

The above general result subsumes a number of special cases which are of interest in their own right:
\begin{enumerate}
\item
 Diagonal correlations on the square lattice $ K_3=0 $,
 \begin{equation}
  \langle \sigma_{0,0}\sigma_{N,N} \rangle^{\square} = \langle \sigma_{0,0}\sigma_{N,N} \rangle^{\triangle} ,
 \label{DiagCorrSL}
 \end{equation}
\item
 Row correlations on the square lattice $ K_1=0 $,
 \begin{equation}
   \langle \sigma_{0,0}\sigma_{N,0} \rangle^{\square}
   = \left. \langle \sigma_{0,0}\sigma_{N,N} \rangle^{\triangle} \right|_{K_3\mapsto K_1} ,
 \label{RowCorrSL}
 \end{equation}
\item
 Column correlations on the square lattice $ K_2=0 $,
 \begin{equation}
   \langle \sigma_{0,0}\sigma_{0,N} \rangle^{\square} 
   = \left. \langle \sigma_{0,0}\sigma_{N,N} \rangle^{\triangle} \right|_{K_3\mapsto K_2} ,
 \label{ColCorrSL}
 \end{equation}
 \item
 Curie Point $ T=T_C $,
 \item
 N\'eel Point $ T=T_N $,
 \item
 Disorder Point $ T=T_D $.
\end{enumerate}

These correlations are governed by a Garnier system because of the simple observation: the weight function is a regular semi-classical 
weight, which is made precise below, and are thus these Toeplitz determinants are known to be $ \tau $-functions for such
equations \cite{Mag_1995a,FW_2006c}.
In fact we are going to heavily employ results from the theory of Garnier system specialised to the {\it classical case}\footnote{
The term classical used in Painlev\'e theory is distinct from that employed in orthogonal polynomial theory 
and refers to the fact that the parameters of the Garnier equation are located on a Weyl chamber wall, i.e. a particular 
integrality condition applies to the parameters and the monodromy matrices are either lower/upper triangular or trivial.},
in the particular setting of bi-orthogonal polynomials and their associated functions defined 
on the unit circle with regular semi-classical weights. The general theory of such 
systems has already been formulated \cite{FW_2006c} for systems with an arbitrary 
number of singularities, following earlier work on orthogonal polynomials on the real line with semi-classical weights in \cite{FIK_1991,Mag_1995a}.

\begin{corollary}[\cite{FW_2006c},\cite{Wit_2009}]\label{GarnierID}
The Toeplitz determinants $ \{ \langle \sigma_{0,0}\sigma_{N,N} \rangle^{\triangle} \}^{\infty}_{N=0} $,
given by \eqref{ssDiag} with \eqref{Telement}, are classical $ \tau $-functions of a Garnier system $ {\mathcal G}_3 $ of type $ L(1^{5};3) $ in three variables formed from the
ratios of any three of $ \{\zeta_j\}^{4}_{j=1} $ to the fourth (assumed non-zero) and whose abbreviated Papperitz-like symbol is
\begin{equation}
      \left\{
\begin{array}{cccccc}
	0 	    & \zeta_1 	    & \zeta_2 	    & \zeta_3 	   & \zeta_4 	  & \infty \\
	\theta_0=N  & \theta_1=-1/2 & \theta_2=-1/2 & \theta_3=1/2 & \theta_4=1/2 & \theta_{\infty}=N 
\end{array}
      \right\} ,
\end{equation}
where the formal monodromy exponent $ \theta_j $ corresponds to the singularity $ \zeta_j $.
\end{corollary}
\begin{proof}
The logarithmic derivative of the weight function \eqref{ssSymbol} has the form
\begin{equation}
   \frac{1}{w(\zeta)}\frac{d}{d\zeta}w(\zeta) = \frac{2V(\zeta)}{W(\zeta)} ,
\label{SCweight}
\end{equation} 
where $ W(\zeta), 2V(\zeta) $ are irreducible polynomials in $ \zeta $. It is clear that 
\begin{gather}
   W(\zeta) = (\zeta-\zeta_1)(\zeta-\zeta_2)(\zeta-\zeta_3)(\zeta-\zeta_4) = \zeta^4-e_1\zeta^3+e_2\zeta^2-e_3\zeta+e_4 ,
\label{m=4_W} \\
  2V(\zeta) = W\sum^{4}_{j=1}\frac{\rho_j}{\zeta-\zeta_j} = m_0\zeta^3-m_1\zeta^2+m_2\zeta-m_3 ,
\label{m=4_2V}
\end{gather}
where the weight data is represented in the following way when the four singularities $ \zeta_j, j=1,2,3,4 $ are in generic positions
and one really has the case of five finite singularities $ M=5 $ as the one at the origin is always present
(as is also the one at infinity under generic circumstances).
Here $ V(z) $, $ W(z) $ are irreducible polynomials that must satisfy the following generic conditions 
of the regular semi-classical class:
\begin{enumerate}
 \item[(i)]
  $ {\rm deg}\;(W) = M \geq 2 $,
 \item[(ii)]
  $ {\rm deg}\;(V) < {\rm deg}\;(W) $,
 \item[(iii)]
  the $ M $ zeros of $ W(z) $, $ \{z_0, \ldots ,z_{M-1}\} $ are pair-wise distinct, and
 \item[(iv)]
  the residues $ \rho_j = 2V(z_j)/W'(z_j) \notin \Z_{\geq 0} $.
\end{enumerate} 
The above generic case defines a system of bi-orthogonal polynomials and their associated functions on the unit circle,
and these satisfy the same second-order linear differential equation which is represented by the following abbreviated
Riemann-Papperitz-like symbol
\begin{equation}
\left\{
\begin{array}{cccccc}
	0 & \zeta_1 & \zeta_2 & \zeta_3 & \zeta_4 & \infty \\
	\theta_0=n & \theta_1=-\rho_1 & \theta_2=-\rho_2 & \theta_3=-\rho_3 & \theta_4=-\rho_4 & \theta_{\infty}=2n+1-\rho_{\infty} 
\end{array}
\right\} .
\end{equation}
Here $ e_4 \neq 0 $ and (\ref{m=4_W}) and (\ref{m=4_2V}) still apply.
In our example $ M=5 $, with four finite, non-zero regular singularities
\begin{equation}
\left\{
\begin{array}{cccccc}
   0 & \zeta_1 & \zeta_2 & \zeta_3 & \zeta_4 & \infty \\
   n & \rho_1=1/2 & \rho_2=1/2 & \rho_3=-1/2 & \rho_4=-1/2 & n
\end{array}
\right\} ,
\end{equation}
and thus one of the conditions on the weight, $ \rho_j \notin \Z $, is satisfied.

The other condition concerns the separation of the singularities $ \zeta_i \neq \zeta_j $ for $ i \neq j $.
The four singularities are governed by a number of parameters: the discriminant $ \mathcal{D} $
\begin{equation}
   \mathcal{D}^2 \coloneqq (z_1+z_2 z_3) (z_2+z_1 z_3) (z_3+z_1 z_2) (1+z_1 z_2 z_3),
\end{equation} 
the additional auxiliary variables
\begin{equation}
  \Gamma \coloneqq 1+v_1^2 v_2^2-(v_1^2+v_2^2) v_3^2 ,\quad 
  \mathcal{D} \coloneqq \frac{4\Delta}{(1+v_1)^2(1+v_2)^2(1+v_3)^2} , \quad
  \bar{\Delta}^2 \coloneqq \frac{(1+v_3^2)^2\Delta^2-4v_3^2\Gamma^2}{(1-v_3^2)^2} ,
\end{equation} 
so that
\begin{equation}
 \Delta^2 \coloneqq (1+v_1 v_2-v_1 v_3-v_2 v_3) (1-v_1 v_2-v_1 v_3+v_2 v_3) (1-v_1 v_2+v_1 v_3-v_2 v_3) (1+v_1 v_2+v_1 v_3+v_2 v_3) ,
\end{equation} 
and
\begin{equation}
  \bar{\Delta}^2 \coloneqq (1-v_1 v_2+v_1 v_3+v_2 v_3) (1+v_1 v_2+v_1 v_3-v_2 v_3) (1+v_1 v_2-v_1 v_3+v_2 v_3) (1-v_1 v_2-v_1 v_3-v_2 v_3) .
\end{equation} 
The singularities have the explicit forms
\begin{equation}
   \zeta_1 = \frac{\Gamma+\Delta}{2v_1v_2(1-v_3)^2} ,\;
   \zeta_2 = \frac{\Gamma-\Delta}{2v_1v_2(1-v_3)^2} ,\;
   \zeta_3 = \frac{\Gamma+\Delta}{2v_1v_2(1+v_3)^2} ,\;
   \zeta_4 = \frac{\Gamma-\Delta}{2v_1v_2(1+v_3)^2} .
\label{singularity}
\end{equation}
Furthermore we note the relations between the singularities
\begin{equation}
  \zeta_1\zeta_2\zeta_3\zeta_4 = 1, \quad
  \zeta_3 = z_3^2 \zeta_1 = \zeta_2^{-1}, \quad \zeta_4 = z_3^2 \zeta_2 = \zeta_1^{-1} .
\end{equation}

The sub-resultants of the numerator and denominator polynomials in \eqref{ssSymbol} are proportional to $ v_1^2 v_2^2 v_3^2 \bar{\Delta}^2 $
and $ v_1 v_2 v_3 \Gamma $. The Curie point \eqref{CuriePt}, the three equivalent variations of the Ne\'el point \eqref{NeelPt}
and the degenerate case of $ T=0 $ account for the vanishing of the first of these sub-resultants, whereby a single cancellation
of a common factor in the numerator and denominator occurs and reduces the number of singularities by two. Examining the 
sub-resultants of the numerator polynomial and its derivative we find $ v_1 v_2 (1-v_3)^2 \Delta^2 $ and $ v_1 v_2 (1-v_3)^2 $.
Proceeding in the same way with the denominator polynomial and its derivative we find $ v_1 v_2 \Delta^2 $ 
and $ v_1 v_2 $. These last two cases include the Disorder point \eqref{DisorderPt} and its three equivalent versions,
as well as the case of $ T=\infty $. Thus in contrast to the generic case
the three critical points and the boundary points $ T = 0 $, $ T = \infty $, correspond to situations where the
singularities coalesce in a pair-wise manner in the following way: \\ 

\noindent
\begin{center}

\begin{tabular}[]{|c|c|}
\hline
 Curie or N\'eel Point &
 \begin{minipage}[c][3cm][c]{7cm}{
    \begin{equation}
       |\zeta_1| = \frac{1}{z_3^2} > |\zeta_2| = |\zeta_3| = 1 >  |\zeta_4| = z_3^2
    \end{equation} 
    or
    \begin{equation}
       |\zeta_2| = \frac{1}{z_3^2} > |\zeta_1| = |\zeta_4| = 1 >  |\zeta_3| = z_3^2
    \end{equation}}
 \end{minipage}
\\ \hline  
 $ T = 0 $ &
 \begin{minipage}[c][1cm][c]{8cm}{
 \begin{equation}
   \zeta_1 = \zeta_3 \gtrless \zeta_2 = \zeta_4 
 \end{equation}}
 \end{minipage}
\\ \hline  
 Disorder Point &  
 \begin{minipage}[c][13mm][c]{8cm}{
 \begin{equation}
   |\zeta_1| = |\zeta_2| = \frac{1}{|z_3|} > 1 > |\zeta_3| = |\zeta_4| = |z_3|
 \end{equation}}
 \end{minipage}
\\ \hline 
 $ T = \infty $ &
 \begin{minipage}[c][1cm][c]{8cm}{
 \begin{equation}
   \zeta_1 = \zeta_3 \gtrless \zeta_2 = \zeta_4 
 \end{equation}}
 \end{minipage}
\\
\hline
\end{tabular}
 
\end{center}

The reason why there are only three independent variables is that the integration contour defining $ w_n $ from the weight
can be contracted or dilated from $ |\zeta| = 1 $ to one of the nearby singularity moduli (keeping the fixed singularities at 
$ \zeta = 0, \infty $ unchanged) which has the effect of normalising the remaining three by this particular one.
\end{proof}

Our first result of the identification of a Garnier system with the spin-spin correlations is a simple recurrence
relation for the Toeplitz matrix elements. 
\begin{corollary}[\cite{Wit_2009}]\label{LinearRR}
The Fourier coefficients or Toeplitz matrix elements $ w_{n} $ satisfy the fourth order, linear homogeneous difference
equation in the index $ n $
\begin{multline}
   (n-3) v_1^2 v_2^2(1-v_3^2)^2 w_{n-3}
   -2 v_1 v_2 \Gamma\left[ v_3+(n-2)(1+v_3^2) \right]w_{n-2}
\\
   + \left[ (n-1) \left(v_1^4+4 v_2^2 v_1^2+v_2^4\right) v_3^4-2 (n-1) \left( v_1^2+v_2^2-6 v_1^2v_2^2+v_1^4v_2^2+v_1^2v_2^4 \right) v_3^2 \right. 
\\
     \left. +(n-1) \left(1+4 v_1^2 v_2^2+v_1^4 v_2^4\right)+8v_1^2 v_2^2 v_3(1+v_3^2) \right]w_{n-1}
\\
   -2 v_1 v_2 \Gamma\left[v_3+n(1+v_3^2)\right]w_{n}
   +(n+1) v_1^2 v_2^2(1-v_3^2)^2 w_{n+1} = 0 .
\label{LRR_triangle}
\end{multline}
In addition for $ n = -1, 3 $ the order drops to third order.
One could take any contiguous set of four elements including $ w_{0} $ as the initial values and iterate in either
direction.
\end{corollary}
\begin{proof}
This follows from an adaptation of Eq. (4.36) of \cite{Wit_2009} to the case where $ e_4 \neq 0 $
\begin{equation}
 \sum_{l=-1}^3 (-1)^{3-l} \left[ (p-l) e_{3-l}-m_{3-l} \right] w_{p-l} = 0, \quad p \in \Z ,
\end{equation}
and the spectral data ($ e_0=1 $)
\begin{align}
   e_1 = e_3 & = -2\frac{(1+v_3^2) \left(\left(v_1^2+v_2^2\right)v_3^2-v_1^2v_2^2-1 \right)}{v_1 v_2 (1-v_3^2)^2} ,
\label{Sdata:1}\\
   e_2 & = \frac{\left(v_1^4+4 v_2^2 v_1^2+v_2^4\right)v_3^4-2 \left(v_2^2 v_1^4+v_2^4v_1^2-6 v_1^2v_2^2+v_1^2+v_2^2\right)v_3^2+v_1^4 v_2^4+4 v_1^2 v_2^2+1}
                {v_1^2 v_2^2 (1-v_3^2)^2} ,
\label{Sdata:2}\\
   e_4 & = 1 ,
\label{Sdata:3}\\
  m_0 & = 0 ,
\label{Sdata:4}\\
  m_1 = m_3 & = 2\frac{v_3 \left(\left(v_1^2+v_2^2\right) v_3^2-v_1^2 v_2^2-1\right)}{v_1 v_2 (1-v_3^2)^2} ,
\label{Sdata:5}\\
  m_2 & = -8\frac{v_3(1+v_3^2)}{(1-v_3^2)^2} .
\label{Sdata:6}
\end{align}
Note that $ w_0 $ cannot be determined by this recurrence by any set of $ w_n $ either $ n<0 $ or $ n>0 $ and is an arbitrary
normalisation of the weight. 
\end{proof}

However the theory of discrete Garnier systems developed in \cite{Wit_2009} furnishes a direct recurrence relation
system for
the Toeplitz determinants themselves, and our second and more significant result is such a system for our application.
Prior to stating these we need define appropriate co-ordinates by choosing one of the singularities, say $ \xi_4 $ 
without any loss of generality 
\begin{equation}
	f^{j}_{n} \equiv \frac{\xi_j}{\xi_4}\frac{\Theta_n(\xi_j)}{\Theta_n(\xi_4)} ,
	\quad g^{j}_{n} = \omega^{j-1}_{n}, \quad j=1,2,3 .
\label{ReccVars}
\end{equation}
The variables $ \Theta_n(z) $ and $ \omega^{j}_{n} $ arise in the parameterisation of the derivatives of the bi-orthogonal
polynomials and associated functions with respect to $ z $, and is termed the {\it spectral structure} of the isomonodromic
system. Further details can be found in \cite{FW_2006c}, \cite{Wit_2009}.

As a preliminary step we need to make the following definitions in order to render the results in the simplest possible form.
\begin{definition}
Let us define four auxiliary variables
\begin{align}
 \mathcal{R}_1 & \coloneqq \frac{(\Gamma-\Delta)}{2 v_1 v_2 (1+v_3)^2}g^{1}_{n}+\left(\frac{\Gamma-\Delta}{2 v_1 v_2 (1+v_3)^2}\right)^2g^{2}_{n}
                   + \left(\frac{\Gamma-\Delta}{2 v_1 v_2 (1+v_3)^2}\right)^3g^{3}_{n}+\left(\frac{\Gamma-\Delta}{2 v_1 v_2 (1+v_3)^2}\right)^4 ,
\label{Rdefn:1}\\
 \mathcal{R}_2 & \coloneqq \frac{(\Gamma+\Delta)}{2 v_1 v_2 (1+v_3)^2}g^{1}_{n}+\left(\frac{\Gamma+\Delta}{2 v_1 v_2 (1+v_3)^2}\right)^2g^{2}_{n}
                   + \left(\frac{\Gamma+\Delta}{2 v_1 v_2 (1+v_3)^2}\right)^3g^{3}_{n}+\left(\frac{\Gamma+\Delta}{2 v_1 v_2 (1+v_3)^2}\right)^4 ,
\label{Rdefn:2}\\
 \mathcal{R}_3 & \coloneqq \frac{(\Gamma+\Delta)}{2 v_1 v_2 (1-v_3)^2}g^{1}_{n}+\left(\frac{\Gamma+\Delta}{2 v_1 v_2 (1-v_3)^2}\right)^2g^{2}_{n}
                   + \left(\frac{\Gamma+\Delta}{2 v_1 v_2 (1-v_3)^2}\right)^3g^{3}_{n}+\left(\frac{\Gamma+\Delta}{2 v_1 v_2 (1-v_3)^2}\right)^4 ,
\label{Rdefn:3}\\
 \mathcal{R}_4 & \coloneqq \frac{(\Gamma-\Delta)}{2 v_1 v_2 (1-v_3)^2}g^{1}_{n}+\left(\frac{\Gamma-\Delta}{2 v_1 v_2 (1-v_3)^2}\right)^2g^{2}_{n}
                   + \left(\frac{\Gamma-\Delta}{2 v_1 v_2 (1-v_3)^2}\right)^3g^{3}_{n}+\left(\frac{\Gamma-\Delta}{2 v_1 v_2 (1-v_3)^2}\right)^4 .
\label{Rdefn:4}
\end{align}
Furthermore we define another set of four auxiliary variables
\begin{multline}
  \mathcal{S}_1 \coloneqq 
        \left( 2\Gamma v_3+\Delta (1+v_3^2) \right) \left[ \frac{(1-v_3)^2}{(1+v_3)^2}(\Gamma+\Delta)-(\Gamma-\Delta) f^{2}_{n} \right]
\\
       +\left( 2\Gamma v_3-\Delta (1+v_3^2) \right) \left[ (\Gamma+\Delta) f^{1}_{n}-\frac{(1-v_3)^2}{(1+v_3)^2}(\Gamma-\Delta) f^{3}_{n} \right] ,
\label{Sdefn:1}
\end{multline}
\begin{multline}
  \mathcal{S}_2 \coloneqq 
        \left( 2\Gamma v_3+\Delta (1+v_3^2) \right) \left[ \frac{(1-v_3)^4}{(1+v_3)^4}(\Gamma+\Delta)^2-(\Gamma-\Delta)^2 f^{2}_{n} \right]
\\
       +\left( 2\Gamma v_3-\Delta (1+v_3^2) \right) \left[ (\Gamma+\Delta)^2 f^{1}_{n}-\frac{(1-v_3)^4}{(1+v_3)^4}(\Gamma-\Delta)^2 f^{3}_{n} \right] ,
\label{Sdefn:2}
\end{multline}
\begin{multline}
  \mathcal{S}_3 \coloneqq 
       \left( 2\Gamma v_3+\Delta (1+v_3^2) \right) \Bigg[ \frac{(1-v_3)^2}{(1+v_3)^2} \left(\Gamma(3v_3^2-2v_3+3)+\Delta(1+v_3)^2\right) (\Gamma+\Delta) \\
              - \left(\Gamma(3v_3^2+2v_3+3)-\Delta(1-v_3)^2\right) (\Gamma-\Delta) f^{2}_{n} \Bigg]
\\
       +\left( 2\Gamma v_3-\Delta (1+v_3^2) \right) \Bigg[ \left(\Gamma(3v_3^2+2v_3+3)+\Delta(1-v_3)^2\right) (\Gamma+\Delta) f^{1}_{n} \\
                          -\frac{(1-v_3)^2}{(1+v_3)^2} \left(\Gamma(3v_3^2-2v_3+3)-\Delta(1+v_3)^2\right) (\Gamma-\Delta) f^{3}_{n}  \Bigg] ,
\label{Sdefn:3}
\end{multline}
\begin{multline}
  \mathcal{S}_4 \coloneqq 
       \left( 2\Gamma v_3+\Delta (1+v_3^2) \right) \Bigg[ \frac{(1-v_3)^4}{(1+v_3)^4} \left(\Gamma(3v_3^2+2v_3+3)-\Delta(1-v_3)^2\right) (\Gamma +\Delta)^2 \\
              - \left(\Gamma(3v_3^2-2v_3+3)+\Delta(1+v_3)^2\right) (\Gamma-\Delta)^2 f^{2}_{n} \Bigg]
\\
       +\left( 2\Gamma v_3-\Delta (1+v_3^2) \right) \Bigg[ \left(\Gamma(3v_3^2-2v_3+3)-\Delta(1+v_3)^2\right) (\Gamma+\Delta)^2 f^{1}_{n} \\
              -\frac{(1-v_3)^4}{(1+v_3)^4} \left(\Gamma(3v_3^2+2v_3+3)+\Delta(1-v_3)^2\right) (\Gamma-\Delta)^2 f^{3}_{n} \Bigg] .
\label{Sdefn:4}
\end{multline} 
\end{definition}

Then one has the following three-variable generalisation of the additive $ D^{(1)}_4 $ member of the Sakai classification
\cite{Sak_2001} otherwise known as the ``fifth discrete Painlev\'e equation'' . Subsequently we will give
explicit relationships between the co-ordinates introduced above and the Toeplitz determinants.
\begin{corollary}[\cite{Wit_2009}]\label{NonlinearRR}
Assume that the singularities $ \zeta_i, \zeta_j $ are pairwise distinct $ i \neq j $, i.e. that $ \Delta \neq 0 $,
$ \bar{\Delta} \neq 0 $ and that $ \mathcal{S}_1, \mathcal{S}_2 \neq 0 $.
The set of variables $ \{ f^{1}_{n},f^{2}_{n},f^{3}_{n},g^{1}_{n},g^{2}_{n},g^{3}_{n} \}^{\infty}_{n=0} $
satisfy the system of coupled, first order non-linear difference equations in $ n \geq 0 $. 
The first set of three constitute the first members of the Lax pair
\begin{equation}
  \frac{(1-v_3)^2}{(1+v_3)^2} f^{1}_{n}f^{1}_{n+1}
   = \frac{\left[ \mathcal{R}_1-n \right]\left[ \mathcal{R}_1-n+\frac{\DySt 2v_3 \Delta}{\DySt v_1^2 v_2^2 (1+v_3)^6}\left( 2\frac{(\DySt \Delta^2-4v_1^2v_2^2v_3 (1+v_3)^2)}
                                                                                                                                   {\DySt \Gamma+\Delta}-\Delta \right) \right]}
          {\left[ \mathcal{R}_4-n \right]\left[ \mathcal{R}_4-n+\frac{\DySt 2v_3 \Delta}{\DySt v_1^2 v_2^2 (1-v_3)^6}\left( 2\frac{(\DySt \Delta^2+4v_1^2v_2^2v_3 (1-v_3)^2)}
                                                                                                                                   {\DySt \Gamma+\Delta}-\Delta \right) \right]} ,
\label{1stLP:1}
\end{equation}
\begin{equation}
  \frac{(1-v_3)^2(\Gamma+\Delta)}{(1+v_3)^2(\Gamma-\Delta)} f^{2}_{n}f^{2}_{n+1}
   = \frac{\left[ \mathcal{R}_2-n \right]\left[ \mathcal{R}_2-n-\frac{\DySt 2v_3 \Delta}{\DySt v_1^2 v_2^2 (1+v_3)^6}\left( 2\frac{\DySt (\Delta^2-4v_1^2v_2^2v_3 (1+v_3)^2)}
                                                                                                                                   {\DySt \Gamma-\Delta}+\Delta \right) \right]}
          {\left[ \mathcal{R}_4-n \right]\left[ \mathcal{R}_4-n+\frac{\DySt 2v_3 \Delta}{\DySt v_1^2 v_2^2 (1-v_3)^6}\left( 2\frac{\DySt (\Delta^2+4v_1^2v_2^2v_3 (1-v_3)^2)}
                                                                                                                                   {\DySt \Gamma+\Delta}-\Delta \right) \right]} ,
\label{1stLP:2}
\end{equation}
\begin{equation}
  \frac{\Gamma+\Delta}{\Gamma-\Delta} f^{3}_{n}f^{3}_{n+1}
   = \frac{\left[ \mathcal{R}_3-n \right]\left[ \mathcal{R}_3-n-\frac{\DySt 2v_3 \Delta}{\DySt v_1^2 v_2^2 (1-v_3)^6}\left( 2\frac{\DySt (\Delta^2+4v_1^2v_2^2v_3 (1-v_3)^2)}
                                                                                                                                   {\DySt \Gamma-\Delta}+\Delta \right) \right]}
          {\left[ \mathcal{R}_4-n \right]\left[ \mathcal{R}_4-n+\frac{\DySt 2v_3 \Delta}{\DySt v_1^2 v_2^2 (1-v_3)^6}\left( 2\frac{\DySt (\Delta^2+4v_1^2v_2^2v_3 (1-v_3)^2)}
                                                                                                                                   {\DySt \Gamma+\Delta}-\Delta \right) \right]} .
\label{1stLP:3}
\end{equation}
The remaining set of coupled, first order non-linear difference equations constitute the three members of the second Lax pair
and are given by $ n \geq 1 $
\begin{multline}
    g^{1}_{n}+g^{1}_{n-1}-\frac{\Gamma}{2v_1v_2(1-v_3)^2} \left[ 2n-1+(2n-3)\frac{(1-v_3)^2}{(1+v_3)^2} \right]
\\
     + \frac{n+1}{2 v_1 v_2(1-v_3)^2} \frac{\mathcal{S}_2}{\mathcal{S}_1}
     - \frac{n}{2v_1v_2(1-v_3^2)^2} \frac{\mathcal{S}_4}{\mathcal{S}_2} = 0 ,
\label{2ndLP:1}
\end{multline}
\begin{multline}
    g^{2}_{n}+g^{2}_{n-1}+\frac{8v_3(1+v_3^2)}{(1-v_3^2)^2}
\\
     + \frac{1+4 v_1^2 v_2^2+v_1^4 v_2^4-2(v_1^2+v_2^2-6v_1^2v_2^2+v_1^4v_2^2+v_1^2v_2^4)v_3^2+(v_1^4+4v_1^2v_2^2+v_2^4)v_3^4}{v_1^2 v_2^2(1-v_3^2)^2}(n-1)
\\
     - \frac{n+1}{4v_1^2v_2^2(1+v_3)^2(1-v_3)^4} \frac{\mathcal{S}_4}{\mathcal{S}_1}
     + \frac{n}{(1+v_3)^2} \frac{\mathcal{S}_3}{\mathcal{S}_2} = 0 ,
\label{2ndLP:2}
\end{multline}
\begin{multline}
    g^{3}_{n}+g^{3}_{n-1}-\frac{\Gamma}{2v_1v_2(1-v_3)^2} \left[ 2n-1+(2n-3)\frac{(1-v_3)^2}{(1+v_3)^2} \right]
\\
     + \frac{n+1}{2v_1 v_2(1-v_3^2)^2} \frac{\mathcal{S}_3}{\mathcal{S}_1}
     - 2nv_1v_2(1-v_3)^2 \frac{\mathcal{S}_1}{\mathcal{S}_2} = 0 .
\label{2ndLP:3}
\end{multline}
These recurrences are subject to the initial values for $ j=1,2,3 $ 
\begin{equation}
     f^{1}_0 = \frac{(1-v_3)^2}{(1+v_3)^2}
	       \frac{\left[ \frac{(\Gamma-\Delta)}{v_1 v_2 (1+v_3)^2}w_{-2}
	                  + \frac{(\Gamma-\Delta)}{4v_1^2v_2^2(1+v_3)^4}\left(\Gamma-\Delta-\frac{4 (1-v_3+v_3^2)}{(1-v_3)^2}\Gamma\right)w_{-1}
	                  - \frac{2v_3\Gamma}{v_1 v_2 (1-v_3^2)^2}w_0 + w_1 \right]}
	            {\left[ \frac{(\Gamma-\Delta)}{v_1 v_2 (1-v_3)^2}w_{-2}
	                  + \frac{(\Gamma-\Delta)}{4v_1^2v_2^2(1-v_3)^4}\left(\Gamma-\Delta-\frac{4 (1-v_3+v_3^2)}{(1+v_3)^2}\Gamma\right)w_{-1}
	                  - \frac{2v_3\Gamma}{v_1 v_2 (1-v_3^2)^2}w_0 + w_1 \right]} ,
\label{initf1} 
\end{equation} 
\begin{multline}
     f^{2}_0 = \frac{(1-v_3)^2(\Gamma+\Delta)}{(1+v_3)^2(\Gamma-\Delta)}
\\ \times
	       \frac{\left[ \frac{(\Gamma+\Delta)}{v_1 v_2 (1+v_3)^2}w_{-2}
	                  + \frac{(\Gamma+\Delta)}{4v_1^2v_2^2(1+v_3)^4}\left(\Gamma+\Delta-\frac{4 (1-v_3+v_3^2)}{(1-v_3)^2}\Gamma\right)w_{-1}
	                  - \frac{2v_3\Gamma}{v_1 v_2 (1-v_3^2)^2}w_0 + w_1 \right]}
	            {\left[ \frac{(\Gamma-\Delta)}{v_1 v_2 (1-v_3)^2}w_{-2}
	                  + \frac{(\Gamma-\Delta)}{4v_1^2v_2^2(1-v_3)^4}\left(\Gamma-\Delta-\frac{4 (1-v_3+v_3^2)}{(1+v_3)^2}\Gamma\right)w_{-1}
	                  - \frac{2v_3\Gamma}{v_1 v_2 (1-v_3^2)^2}w_0 + w_1 \right]} ,
\label{initf2} 
\end{multline}
\begin{equation}
     f^{3}_0 = \frac{(\Gamma+\Delta)}{(\Gamma-\Delta)}
	       \frac{\left[ \frac{(\Gamma+\Delta)}{v_1 v_2 (1-v_3)^2}w_{-2}
	                  + \frac{(\Gamma+\Delta)}{4v_1^2v_2^2(1-v_3)^4}\left(\Gamma+\Delta-\frac{4 (1-v_3+v_3^2)}{(1+v_3)^2}\Gamma\right)w_{-1}
	                  - \frac{2v_3\Gamma}{v_1 v_2 (1-v_3^2)^2}w_0 + w_1 \right]}
	            {\left[ \frac{(\Gamma-\Delta)}{v_1 v_2 (1-v_3)^2}w_{-2}
	                  + \frac{(\Gamma-\Delta)}{4v_1^2v_2^2(1-v_3)^4}\left(\Gamma-\Delta-\frac{4 (1-v_3+v_3^2)}{(1+v_3)^2}\Gamma\right)w_{-1}
	                  - \frac{2v_3\Gamma}{v_1 v_2 (1-v_3^2)^2}w_0 + w_1 \right]} ,
\label{initf3} 
\end{equation} 
and
\begin{align}
	g^{1}_0 & = 2\frac{v_3\Gamma}{v_1v_2(1-v_3^2)^2}-\frac{w_1}{w_0} ,
\label{initg1}\\
	g^{2}_0 & = - 8\frac{v_3(1+v_3^2)}{(1-v_3^2)^2} - 2\frac{v_3\Gamma}{v_1 v_2 (1-v_3^2)^2}\frac{w_0}{w_{-1}} 
	            + 2\frac{(1+v_3+v_3^2)\Gamma}{v_1 v_2 (1-v_3^2)^2}\frac{w_1}{w_0} + \frac{w_1}{w_{-1}} - 2\frac{w_2}{w_0},
\label{initg2}\\
	g^{3}_0 & = 2\frac{w_{-2}}{w_{-1}} - \frac{w_{-1}}{w_0} - 2\frac{(1-v_3+v_3^2)\Gamma}{v_1 v_2(1-v_3^2)^2} .
\label{initg3} 
\end{align}
\end{corollary}
\begin{proof}
We will first permute and re-label the singularities in the following way $ \xi_1=\zeta_4 $, $ \xi_2=\zeta_3 $, $ \xi_3=\zeta_1 $,
$ \xi_4=\zeta_2 $ so that we can rescale the contour radius in order to normalise with respect to $ \xi_4 $.
From an adaptation of the first result in Proposition 4.3 of \cite{Wit_2009} the first Lax pair can be written for $ j=1,2,3 $ as
\begin{multline}
 \frac{\xi_j}{\xi_4} f^{j}_{n} f^{j}_{n+1}
 = \frac{\left[ \xi_j g^{1}_{n}+\xi_j^2 g^{2}_{n}+\xi_j^3 g^{3}_{n}+\xi_j^4 \left(1+m_0\right)-n e_4 \right]}
        {\left[ \xi_4 g^{1}_{n}+\xi_4^2 g^{2}_{n}+\xi_4^3 g^{3}_{n}+\xi_4^4 \left(1+m_0\right)-n e_4 \right]}
\\ \times
   \frac{\left[ \xi_j \left(g^{1}_{n}+m_3\right)+\xi_j^2 \left(g^{2}_{n}-m_2\right)+\xi_j^3 \left(g^{3}_{n}+m_1\right)+\xi_j^4-n e_4 \right]}
        {\left[ \xi_4 \left(g^{1}_{n}+m_3\right)+\xi_4^2 \left(g^{2}_{n}-m_2\right)+\xi_4^3 \left(g^{3}_{n}+m_1\right)+\xi_4^4-n e_4 \right]} ,
\end{multline}
Substituting the evaluations of the singular points \eqref{singularity} and employing the definitions
\eqref{Rdefn:1}-\eqref{Rdefn:4} we simplify the resulting expressions and arrive at \eqref{1stLP:1}-\eqref{1stLP:3}. 

Proceeding we now define $ s_{n}[T=\{a,b,\ldots,\}] \equiv \sum_{i_n > \ldots > i_1 \in T} \xi_{i_1} \ldots \xi_{i_n} $ as the $n$-th
elementary symmetric function in the variables $ \xi_a, \xi_b, \ldots $ and the Vandermode determinant 
$ \Delta[T=\{a,b,\dots\}] \equiv \prod_{r<s \in \{a,b,\ldots\}}(\xi_s-\xi_r) $. Adapting now the second result in Proposition 4.3
of \cite{Wit_2009} the second member of the Lax pair is given for $ j=1,2,3 $ by
\begin{multline}
  g^{j}_{n}+g^{j}_{n-1}+(-1)^{j}(n-1)e_{4-j}-(-1)^{j}m_{4-j}
\\
   +(-1)^{j+1}(n+1+m_0)\frac{ s_{4-j}[123]\Delta[123] - s_{4-j}[234]\Delta[234] f^{1}_{n} + s_{4-j}[134]\Delta[134] f^{2}_{n} - s_{4-j}[124]\Delta[124] f^{3}_{n} }
         { \Delta[123] - \Delta[234] f^{1}_{n} + \Delta[134] f^{2}_{n} - \Delta[124] f^{3}_{n} }
\\
   +(-1)^{j}n e_4\frac{ s_{3-j}[123]\Delta[123] - s_{3-j}[234]\Delta[234] f^{1}_{n} + s_{3-j}[134]\Delta[134] f^{2}_{n} - s_{3-j}[124]\Delta[124] f^{3}_{n} }
         { s_{3}[123]\Delta[123] - s_{3}[234]\Delta[234] f^{1}_{n} + s_{3}[134]\Delta[134] f^{2}_{n} -  s_{3}[124]\Delta[124] f^{3}_{n} } = 0 .
\end{multline}
Using the definitions \eqref{Sdefn:1}-\eqref{Sdefn:4} and the evaluations of the other data we deduce \eqref{2ndLP:1}-\eqref{2ndLP:3}.

The first set of initial values, \eqref{initf1}-\eqref{initf3}, are ratios of the spectral coefficient (see Eqs. (2.6) and (2.10) of \cite{Wit_2009})
\begin{equation}
  \Theta_n(\zeta) = -ne_4\frac{\phi_n(0)}{\phi_{n+1}(0)}\zeta^{-1} + \frac{\kappa_n}{\kappa_{n+1}}\vartheta^{0}_{n} + \frac{\kappa_n}{\kappa_{n+1}}\vartheta^{1}_{n}\zeta
  + \frac{\kappa_n}{\kappa_{n+1}}(n+1+m_0)\zeta^2 ,
\end{equation}
appropriately specialised using the data \eqref{Sdata:1}-\eqref{Sdata:6} and evaluated at $ \zeta = \xi_j $ for $ j=1,\ldots, 4 $. 
The second set, \eqref{initg1}-\eqref{initg3}, can be derived from a third spectral polynomial $ U(\zeta) $ using the generic relation
\begin{equation}
   \Omega_0(\zeta)+V(\zeta) = \tfrac{1}{2}\left( 1+\frac{1}{r_1}\zeta \right)\left[ 2V(\zeta)-\kappa^2_0U(\zeta) \right] ,
\end{equation} 
and the fact that $ g^{l+1}_{n} = [\zeta^{l}](\Omega_0+V) $, which follows from the parameterisation of $ \Omega_n $. The polynomial
$ U $ has the explicit form (which is an adaptation of Eqs. (4.37)-(4.39) of \cite{Wit_2009} for $ e_4\neq 0 $) with coefficients
\begin{gather}
   u_0 = 2e_4w_1+m_3w_0, \quad u_3 = m_0w_0 ,
\\
   u_l = (-1)^{4-l}m_{3-l}w_0+2\sum^{\min(l-1,3)}_{r=-1} (-1)^{3-r}\left[ (l-r)e_{3-r}-m_{3-r} \right]w_{l-r} ,\quad l = 1, 2,
\end{gather}
in terms of a contiguous set of initial moments, essentially the same set that defines the initial values of the recurrence in
Corollary \ref{LinearRR}. Thus $ U $ encodes these initial values in an alternative way.
\end{proof}

Under the conditions applying to Corollary \ref{NonlinearRR}
the primary variables can be recovered from the $ \{f^{j}_n, g^{j}_n \} $ variables using the formulae given in
the following corollary and is a direct consequence of the theory in \cite{FW_2006c}, \cite{Wit_2009}.
\begin{corollary}[\cite{FW_2006c}, \cite{Wit_2009}]\label{Recover}
Assume the conditions of Corollary \ref{NonlinearRR}.
The Toeplitz determinants \eqref{ssDiag} and \eqref{Telement}, abbreviated as $ \langle \sigma_{0,0}\sigma_{N,N} \rangle^{\triangle} \eqqcolon I_{N} $, are computed
using the recurrence relation equation
\begin{equation}
  \frac{I_{n+1}I_{n-1}}{I^2_{n}} = 1-r_n\bar{r}_n , \quad n \geq 1 ,
\label{Irecur}
\end{equation}
with the initial conditions $ I_0=1 $, $ I_1=w_0 $. In turn the pair of reflection coefficients $ r_n, \bar{r}_n $ are computed firstly using
the recurrence relation for $ r_n $
\begin{equation}
       \frac{r_{n+1}}{r_n} = \frac{2n v_1 v_2(1-v_3)^2}{n+1} \frac{\mathcal{S}_1}{\mathcal{S}_2} , \quad n \geq 1.
\label{recover:1}
\end{equation} 
The coefficient $ \bar{r}_n $ is computed via another coefficient $ \lambda_{n} $ which satisfies
\begin{equation}
     (n+1)\lambda_{n+1}-n\lambda_{n} = -2n\frac{(1+v_3^2)\Gamma}{v_1 v_2(1-v_3^2)^2} + g^{3}_n + \frac{n+1}{2v_1 v_2(1-v_3^2)^2} \frac{\mathcal{S}_3}{\mathcal{S}_1} ,
\label{recover:2}
\end{equation} 
and the generic relation $ \lambda_{n+1}-\lambda_{n} = r_{n+1}\bar{r}_{n} $.
Alternatively one can use
\begin{equation}
   (n+1)\bar{\lambda}_{n+1}-n\bar{\lambda}_{n} = g^{1}_n - \frac{\Gamma}{2v_1 v_2(1-v_3)^2} \left[ 2n+1+(2n-1)\frac{(1-v_3)^2}{(1+v_3)^2} \right]
        + \frac{(n+1)}{2v_1v_2(1-v_3)^2}\frac{\mathcal{S}_2}{\mathcal{S}_1} ,
\label{recover:3}
\end{equation}
along with $ \bar{\lambda}_{n}-\bar{\lambda}_{n-1} = \bar{r}_{n}r_{n-1} $ to determine $ \bar{r}_n $.
The initial conditions are $ r_0 = \bar{r}_0 = 1 $ and $ \lambda_0 = \bar{\lambda}_0 = 0 $.
\end{corollary}
\begin{proof}
From the general theory the {\it spectral coefficients} can be parameterised, in the case of four finite singularities, in the
following forms (see Eqs. (2.6) and (2.10) of \cite{Wit_2009})
\begin{equation}
  \Theta_n = -ne_4\frac{\phi_n(0)}{\phi_{n+1}(0)}z^{-1} + \frac{\kappa_n}{\kappa_{n+1}}\vartheta^{0}_{n} + \frac{\kappa_n}{\kappa_{n+1}}\vartheta^{1}_{n}z
  + \frac{\kappa_n}{\kappa_{n+1}}(n+1+m_0)z^2 ,
\end{equation}
and (see Eqs. (2.7) and (2.11) of \cite{Wit_2009})
\begin{equation} 
  \Omega_n = -ne_4z^{-1} + \omega^{0}_n+\tfrac{1}{2}m_3 + (\omega^{1}_n-\tfrac{1}{2}m_2)z + (\omega^{2}_n+\tfrac{1}{2}m_1)z^2 + (1+\tfrac{1}{2}m_0)z^3 .
\end{equation}
The parameters introduced above can, in turn, be related to the leading or trailing coefficients (with respect to $ z $) of
the bi-orthogonal polynomials via the formulae
\begin{gather}
   \frac{r_{n+1}}{r_{n}}\vartheta^{0}_n = ne_3-m_3
	   +e_4\left[(n+1)\bar{\lambda}_{n+1}-(n-1)\left(\bar{\lambda}_{n-1}+\frac{r_{n-1}}{r_{n}}\right)\right] ,
\\
   \vartheta^{1}_n = -(n+1)e_1-m_1+(n+2+m_0)\left[\frac{r_{n+2}}{r_{n+1}}-\lambda_{n+2}\right]+(n+m_0)\lambda_n ,
\label{Aux:2}\\
   \omega^{0}_n = ne_3-m_3+e_4\left[ (n+1)\bar{\lambda}_{n+1}-n\left(\bar{\lambda}_n+\frac{r_n}{r_{n+1}}\right) \right] ,
\label{Aux:3}\\
   \omega^{2}_n = -e_1-m_1+(n+1+m_0)\lambda_{n+1}-(n+2+m_0)\left[ \lambda_{n+2}-\frac{r_{n+2}}{r_{n+1}} \right] .
\label{Aux:4}
\end{gather}
However these parameters are also given by the dynamical co-ordinates by the formulae
\begin{equation}
    \frac{r_n}{r_{n+1}} = \frac{n+1+m_0}{n} 
	\frac{ s_3[123]\Delta[123] - s_3[234]\Delta[234] f^{1}_n + s_3[134]\Delta[134] f^{2}_n - s_3[124]\Delta[124] f^{3}_n }
	     { \Delta[123] - \Delta[234] f^{1}_n + \Delta[134] f^{2}_n - \Delta[124] f^{3}_n } ,
\label{tmp:1}
\end{equation}
\begin{equation}
  \vartheta^{0}_n = (n+1+m_0)
	\frac{ s_2[123]\Delta[123] - s_2[234]\Delta[234] f^{1}_n + s_2[134]\Delta[134] f^{2}_n - s_2[124]\Delta[124] f^{3}_n }
	     { \Delta[123] - \Delta[234] f^{1}_n + \Delta[134] f^{2}_n - \Delta[124] f^{3}_n } ,
\label{tmp:2}
\end{equation}
and
\begin{equation}
  \vartheta^{1}_n = -(n+1+m_0)
	\frac{ s_1[123]\Delta[123] - s_1[234]\Delta[234] f^{1}_n + s_1[134]\Delta[134] f^{2}_n - s_1[124]\Delta[124] f^{3}_n }
	     { \Delta[123] - \Delta[234] f^{1}_n + \Delta[134] f^{2}_n - \Delta[124] f^{3}_n } .
\label{tmp:3}
\end{equation}
Combining \eqref{Aux:2} and \eqref{Aux:4} we note
\begin{equation}
   \omega^{2}_n - \vartheta^{1}_n = ne_1+(n+1+m_0)\lambda_{n+1}-(n+m_0)\lambda_{n} . 
\label{Aux:5}
\end{equation}
The first of these, \eqref{tmp:1}, gives the formula \eqref{recover:1}.
From the preceding relation \eqref{Aux:5} and \eqref{tmp:3} we deduce the formula \eqref{recover:2}. There are also other ways
to determine $ \bar{r}_{n} $, such as using \eqref{Aux:3} with \eqref{tmp:1}, and this yields \eqref{recover:3}. In addition
to the forgoing specific relations we have the general identities
\begin{equation}
    \lambda_{n+1}-\lambda_{n} = r_{n+1}\bar{r}_{n} ,\quad \bar{\lambda}_{n+1}-\bar{\lambda}_{n} = \bar{r}_{n+1}r_{n} ,
\end{equation}
which allow us to close the system.
Finally the Toeplitz determinants satisfy the second order difference equation
\begin{equation}
  \frac{I_{n+1}I_{n-1}}{I^2_{n}} = 1-r_n\bar{r}_n ,
\end{equation}
which is also a general identity.
\end{proof}

In the remaining part of our study we discuss two important special cases that arise from the general formulae above
by taking appropriate limits as given in \eqref{ColCorrSL} and \eqref{DiagCorrSL}.

\subsection{Correlations along a Row/Column of the Rectangular Ising Model}
The row and column correlations of the rectangular Ising model also have a Toeplitz determinant form
(see VIII.2.28,29,30 of \cite{MW_1973}). Because the row and column correlations are given by each other under the
exchange $ z_1 \leftrightarrow z_2 $ we will treat only the column correlations without any loss of generality.
These latter correlations are given by
\begin{gather} 
  \langle\sigma_{0,0}\sigma_{0,N}\rangle^{\square} = \det(w_{j-k})_{0\leq j,k\leq N-1}, 
\label{ssColumn} 
\end{gather} 
where the Fourier coefficients $ w_{n} $ 
\begin{equation}
   w_{n} = \int^{\pi}_{-\pi}\frac{d\theta}{2\pi}
           \frac{\left(1+\alpha_1\alpha_2-(\alpha_1+\alpha_2)\cos(\theta)\right)\cos(n\theta)-(\alpha_1-\alpha_2)\sin(\theta)\sin(n\theta)}
                {\sqrt{\left(1-2\alpha_1\cos(\theta)+\alpha_1^2\right)\left(1-2\alpha_2\cos(\theta)+\alpha_2^2\right)}} ,
\label{ToepCol}
\end{equation} 
are defined by the weight
\begin{equation} 
  w(\zeta) = \left[
  \frac{(1-\alpha_1\zeta)}{(1-\alpha_1\zeta^{-1})}
  \frac{(1-\alpha_2\zeta^{-1})}{(1-\alpha_2\zeta)}
                     \right]^{1/2}.
\end{equation}
Two new co-ordinates are defined by
\begin{equation}
   \alpha_1 = z_2\frac{1-z_1}{1+z_1},\quad \alpha_2 = \frac{1}{z_2}\frac{1-z_1}{1+z_1} .
\end{equation} 
The parameters $\alpha_1,\alpha_2$ are also related to the set $ S, C $ and $ \bar{S}, \bar{C} $ by
\begin{equation}
  \alpha_1 = \frac{\bar{C}-1}{\bar{S}}(C-S), \qquad
  \alpha_2 = \frac{\bar{S}}{\bar{C}-1}(C-S).
\end{equation} 
which is entirely analogous to the definitions of McCoy and Wu but differs because of conventions for the lattice co-ordinate system.
The modulus is related to these parameters by
\begin{equation}
  k = \frac{1-\alpha_1\alpha_2}{\alpha_2-\alpha_1}.
\end{equation}
For $ 0<k<\infty $ the parameters satisfy the inequalities
\begin{align}
  0 < \alpha_1 \leq \alpha_2 \leq 1 \leq \alpha_2^{-1} \leq \alpha_1^{-1}
  & \quad k>1 , \; T < T_C , \\
  0 < \alpha_1 \leq \alpha_2^{-1} \leq 1 \leq \alpha_2 \leq \alpha_1^{-1}
  & \quad k<1 ,\; T > T_C .  
\end{align}

\begin{lemma}
The Fourier coefficients satisfy the symmetry
\begin{equation}
   w_n(\alpha_1,\alpha_2) = w_{-n}(\alpha_2,\alpha_1)
\end{equation}
\end{lemma}

As indicated by \eqref{ColCorrSL} the result \eqref{ssColumn} and \eqref{ToepCol} can be found from the diagonal correlations
on the triangular lattice by setting $ K_2\to 0 $ followed by
the relabelling $ K_3 \to K_2 $. This is still within the $ M=4 $ semi-classical class of the general case even though
the number of independent variables has been reduced by one. The singular points are now
\begin{equation}
   (\zeta_1, \zeta_2, \zeta_3, \zeta_4) = (\alpha_1^{-1},\alpha_2,\alpha_2^{-1},\alpha_1) .
\end{equation}
Even though this is merely a specialisation within the original system we will record the full details of the final 
recurrence relations for the convenience of the reader since they are new and of separate and intrinsic interest.
\begin{corollary}\label{LinearRR_ColCorrSqL}
The Toeplitz matrix elements $ w_{n} $ for the column correlations of the anisotropic square lattice Ising model satisfy
the fourth order, linear homogeneous difference equation in the index $ n \in  \Z $
\begin{multline}
      2\alpha_1\alpha_2(n-3)w_{n-3} - (1+\alpha_1 \alpha_2)\left[ (2n-5)\alpha_1+(2n-3)\alpha_2 \right] w_{n-2}
\\
    + 2\left[ (n-2)\alpha_1^2+n\alpha_2^2+(n-1)(1+\alpha_1\alpha_2)^2 \right]w_{n-1}
\\
    - (1+\alpha_1 \alpha_2)\left[ (2n-1)\alpha_1+(2n+1)\alpha_2 \right]w_{n} + 2\alpha_1\alpha_2(n+1) w_{n+1} = 0 .
\end{multline}
One could take any contiguous set of four elements including $ w_{0} $ as the initial values and iterate in either
direction.
\end{corollary}

And the corresponding specialisation of the coupled non-linear recurrences for the Toeplitz determinants is given by the
following result.
\begin{corollary}\label{NonlinearRR_ColCorrSqL}
Assume that
$ (1-\alpha_1^2)\left[ 1-\alpha_2^2 f^{2}_{n} \right] - (1-\alpha_2^2)\left[ f^{1}_{n}-\alpha_1^2f^{3}_{n} \right] \neq 0 $ and
$ \alpha_1 (1-\alpha_1^2)\left[ 1-\alpha_2^4f^{2}_{n} \right] - \alpha_2(1-\alpha_2^2)\left[ f^{1}_{n}-\alpha_1^4f^{3}_{n} \right] \neq 0 $.
The set of variables $ \{ f^{1}_{n},f^{2}_{n},f^{3}_{n},g^{1}_{n},g^{2}_{n},g^{3}_{n} \}^{\infty}_{n=0} $ for the column
correlations of the anisotropic square lattice Ising model satisfy a system of coupled, first order non-linear difference
equations in $ n \geq 0 $. The first set of three constitute the first members of the Lax pair
\begin{multline}
   \frac{\alpha_1}{\alpha_2}f^{1}_{n}f^{1}_{n+1} =
\\
   \frac{ \left[ \alpha_1 g^{1}_{n}+\alpha_1^2 g^{2}_{n}+\alpha_1^3 g^{3}_{n}+\alpha_1^4-n \right] }
        { \left[ \alpha_2 g^{1}_{n}+\alpha_2^2 g^{2}_{n}+\alpha_2^3 g^{3}_{n}+\alpha_2^4-n \right] }
   \frac{ \left[ \alpha_1 g^{1}_{n}+\alpha_1^2 g^{2}_{n}+\alpha_1^3 g^{3}_{n}-n-\frac{1}{2}+\frac{3}{2}\alpha_1^4+\frac{\alpha_1}{2\alpha_2}(1-\alpha_1^2)(1+\alpha_2^2) \right] }
        { \left[ \alpha_2 g^{1}_{n}+\alpha_2^2 g^{2}_{n}+\alpha_2^3 g^{3}_{n}-n+\frac{1}{2}+\frac{1}{2}\alpha_2^4-\frac{\alpha_2}{2\alpha_1}(1+\alpha_1^2)(1-\alpha_2^2) \right] } ,
\label{1stLP_ColCorrSL:1}
\end{multline}
\begin{multline}
   \alpha_2^6 f^{2}_{n}f^{2}_{n+1} =
\\
   \frac{ \left[ \alpha_2^3 g^{1}_{n}+\alpha_2^2 g^{2}_{n}+\alpha_2 g^{3}_{n}-n\alpha_2^4+1 \right] }
        { \left[ \alpha_2 g^{1}_{n}+\alpha_2^2 g^{2}_{n}+\alpha_2^3 g^{3}_{n}+\alpha_2^4-n \right] }
   \frac{ \left[ \alpha_2^3 g^{1}_{n}+\alpha_2^2 g^{2}_{n}+\alpha_2 g^{3}_{n}-(n+\frac{1}{2})\alpha_2^4+\frac{3}{2}-\frac{\alpha_2}{2\alpha_1}(1+\alpha_1^2)(1-\alpha_2^2) \right] }
        { \left[ \alpha_2 g^{1}_{n}+\alpha_2^2 g^{2}_{n}+\alpha_2^3 g^{3}_{n}-n+\frac{1}{2}+\frac{1}{2}\alpha_2^4-\frac{\alpha_2}{2\alpha_1}(1+\alpha_1^2)(1-\alpha_2^2) \right] } ,
\label{1stLP_ColCorrSL:2}
\end{multline}
\begin{multline}
   \frac{\alpha_1^7}{\alpha_2} f^{3}_{n}f^{3}_{n+1} = 
\\
   \frac{ \left[ \alpha_1^3 g^{1}_{n}+\alpha_1^2 g^{2}_{n}+\alpha_1 g^{3}_{n}-n\alpha_1^4+1 \right] }
        { \left[ \alpha_2 g^{1}_{n}+\alpha_2^2 g^{2}_{n}+\alpha_2^3 g^{3}_{n}+\alpha_2^4-n \right] }
   \frac{ \left[ \alpha_1^3 g^{1}_{n}+\alpha_1^2 g^{2}_{n}+\alpha_1 g^{3}_{n}-(n-\frac{1}{2})\alpha_1^4+\frac{1}{2}+\frac{\alpha_1}{2\alpha_2}(1-\alpha_1^2)(1+\alpha_2^2)  \right] }
        { \left[ \alpha_2 g^{1}_{n}+\alpha_2^2 g^{2}_{n}+\alpha_2^3 g^{3}_{n}-n+\frac{1}{2}+\frac{1}{2}\alpha_2^4-\frac{\alpha_2}{2\alpha_1}(1+\alpha_1^2)(1-\alpha_2^2) \right] } .
\label{1stLP_ColCorrSL:3}
\end{multline}
The remaining set of coupled, first order non-linear difference equations constitute the three members of the second Lax pair
and are given by
\begin{multline}
  g^{1}_{n}+ g^{1}_{n-1}-\frac{(1+\alpha_1\alpha_2)}{2\alpha_1\alpha_2}\left[ (2n-3)\alpha_1+(2n-1)\alpha_2 \right]
\\
  +\frac{1}{\alpha_1\alpha_2}(n+1)
   \frac{\left[ \alpha_1 (1-\alpha_1^2)\left[ 1-\alpha_2^4f^{2}_{n} \right] - \alpha_2 (1-\alpha_2^2)\left[ f^{1}_{n}-\alpha_1^4f^{3}_{n} \right] \right]}
        {\left[ (1-\alpha_1^2)\left[ 1-\alpha_2^2 f^{2}_{n} \right] - (1-\alpha_2^2)\left[ f^{1}_{n}-\alpha_1^2f^{3}_{n} \right] \right]}
\\ 
   -n
   \frac{\left[ \splitdfrac{ (1-\alpha_1^2)\left[ (1+\alpha_1 (\alpha_1+\alpha_2))-(\alpha_1+\alpha_2(1+\alpha_1^2))\alpha_2^3f^{2}_{n} \right] }
                    { - (1-\alpha_2^2)\left[ (1+\alpha_2 (\alpha_1+\alpha_2))f^{1}_{n}-(\alpha_2+\alpha_1(1+\alpha_2^2))\alpha_1^3f^{3}_{n} \right]} \right]}
        {\left[ \alpha_1 (1-\alpha_1^2)\left[ 1-\alpha_2^4f^{2}_{n} \right] - \alpha_2(1-\alpha_2^2)\left[ f^{1}_{n}-\alpha_1^4f^{3}_{n} \right] \right]} = 0 ,
\label{2ndLP_ColCorrSL:1}
\end{multline}
\begin{multline}
   g^{2}_{n}+g^{2}_{n-1}+\frac{\alpha_2^2-\alpha_1^2}{\alpha_1\alpha_2}+(n-1)\frac{1}{\alpha_1\alpha_2}\left[ \alpha_1^2+\alpha_2^2+(1+\alpha_1\alpha_2)^2 \right]
\\
   +n
   \frac{\left[ \splitdfrac{ (1-\alpha_1^2)\left[ (\alpha_1+\alpha_2(1+\alpha_1^2))-(1+\alpha_1(\alpha_1+\alpha_2))\alpha_2^3f^{2}_{n} \right]}
                    {- (1-\alpha_2^2)\left[ (\alpha_2+\alpha_1(1+\alpha_2^2))f^{1}_{n}-(1+\alpha_2(\alpha_1+\alpha_2))\alpha_1^3f^{3}_{n} \right]} \right]}
        {\left[ \alpha_1(1-\alpha_1^2)\left[ 1-\alpha_2^4f^{2}_{n} \right] - \alpha_2 (1-\alpha_2^2)\left[ f^{1}_{n}-\alpha_1^4f^{3}_{n} \right] \right]}
\\
   -\frac{n+1}{\alpha_1 \alpha_2}
   \frac{\left[ \splitdfrac{ (1-\alpha_1^2)\left[ (1+\alpha_1(\alpha_1+\alpha_2))-(\alpha_1+\alpha_2(1+\alpha_1^2))\alpha_2^3f^{2}_{n} \right]}
                    { - (1-\alpha_2^2)\left[ (1+\alpha_2 (\alpha_1+\alpha_2))f^{1}_{n}-(\alpha_2+\alpha_1(1+\alpha_2^2))\alpha_1^3f^{3}_{n} \right]} \right]}
        {\left[ (1-\alpha_1^2)\left[ 1-\alpha_2^2 f^{2}_{n} \right] - (1-\alpha_2^2)\left[ f^{1}_{n}-\alpha_1^2f^{3}_{n} \right] \right]} = 0 ,
\label{2ndLP_ColCorrSL:2}
\end{multline}
\begin{multline}
   g^{3}_{n}+g^{3}_{n-1}-\frac{(1+\alpha_1\alpha_2)}{2\alpha_1\alpha_2}\left[ (2n-3)\alpha_1+(2n-1)\alpha_2 \right]
\\
   -\alpha_1\alpha_2 n
   \frac{\left[ (1-\alpha_1^2)\left[ 1-\alpha_2^2f^{2}_{n} \right] - (1-\alpha_2^2)\left[ f^{1}_{n}-\alpha_1^2f^{3}_{n} \right] \right]}
        {\left[ \alpha_1 (1-\alpha_1^2)\left[ 1-\alpha_2^4f^{2}_{n} \right] - \alpha_2(1-\alpha_2^2)\left[ f^{1}_{n}-\alpha_1^4f^{3}_{n} \right] \right]}
\\
   +\frac{n+1}{\alpha_1\alpha_2}
   \frac{\left[ \splitdfrac{ (1-\alpha_1^2)\left[ (\alpha_1+\alpha_2(1+\alpha_1^2))-(1+\alpha_1(\alpha_1+\alpha_2))\alpha_2^3f^{2}_{n} \right]}
                    {-(1-\alpha_2^2)\left[ (\alpha_2+\alpha_1(1+\alpha_2^2))f^{1}_{n}-(1+\alpha_2(\alpha_1+\alpha_2))\alpha_1^3f^{3}_{n} \right]} \right]}
        {\left[ (1-\alpha_1^2)\left[ 1-\alpha_2^2f^{2}_{n} \right]-(1-\alpha_2^2)\left[ f^{1}_{n}-\alpha_1^2f^{3}_{n} \right] \right]} = 0 .
\label{2ndLP_ColCorrSL:3}
\end{multline}
These recurrences are subject to the initial values for $ j=1,2,3 $ 
\begin{equation}
   f^{1}_{0} =  \frac {\alpha_1}{\alpha_2}
                \frac{ 4\alpha_1^2\alpha_2 w_{-2} - \alpha_1(\alpha_2(1+\alpha_1^2)+\alpha_1(3+\alpha_2^2)) w_{-1} - (\alpha_2(1-\alpha_1^2)-\alpha_1(1-\alpha_2^2)) w_0 + 2\alpha_1\alpha_2 w_1}
                     { 4\alpha_1\alpha_2^2 w_{-2} - \alpha_2(\alpha_2(1+3\alpha_1^2)+\alpha_1(3-\alpha_2^2)) w_{-1} + (\alpha_1(1-\alpha_2^2)-\alpha_2(1-\alpha_1^2)) w_0 + 2\alpha_1\alpha_2 w_1} ,
\end{equation} 
\begin{equation}
   f^{2}_{0} =  \frac{1}{\alpha_2^3}
                \frac{ 4\alpha_1\alpha_2 w_{-2} - (\alpha_1+\alpha_2+\alpha_1\alpha_2(3\alpha_1+\alpha_2)) w_{-1} + \alpha_2(\alpha_1-\alpha_2)(1+\alpha_1\alpha_2) w_0 + 2\alpha_1\alpha_2^2 w_1}
                     { 4\alpha_1\alpha_2^2 w_{-2} - \alpha_2(3\alpha_1+\alpha_2+\alpha_1\alpha_2(3\alpha_1-\alpha_2)) w_{-1} + (\alpha_1-\alpha_2) (1+\alpha_1\alpha_2) w_0 + 2\alpha_1\alpha_2 w_1} ,
\end{equation}
\begin{equation}
   f^{3}_{0} = \frac{1}{\alpha_1^2\alpha_2}
               \frac{ 4\alpha_2\alpha_1 w_{-2} - (3\alpha_1-\alpha_2+\alpha_1\alpha_2(3\alpha_1+\alpha_2)) w_{-1} + \alpha_1(\alpha_1-\alpha_2)(1+\alpha_1\alpha_2) w_0 + 2\alpha_1^2\alpha_2 w_1}
                    { 4\alpha_1\alpha_2^2 w_{-2} - \alpha_2(3\alpha_1+\alpha_2+\alpha_1\alpha_2(3\alpha_1-\alpha_2))w_{-1} + (\alpha_1-\alpha_2)(1+\alpha_1\alpha_2) w_0 + 2\alpha_1\alpha_2 w_1} ,
\end{equation} 
and
\begin{equation}
   g^{1}_{0} = -\frac{(\alpha_1(1-\alpha_2^2)-\alpha_2(1-\alpha_1^2)) w_0 + 2\alpha_1\alpha_2 w_1}
                     {2\alpha_1\alpha_2 w_0} ,
\end{equation}
\begin{equation}
   g^{2}_{0} = \frac{\left[ \splitdfrac{ 2(\alpha_1^2-\alpha_2^2) w_{-1}w_0 + (\alpha_1+3\alpha_2)(1+\alpha_1\alpha_2) w_{-1}w_1 }
                                { - 4\alpha_1\alpha_2 w_{-1}w_2 + (\alpha_1-\alpha_2)(1+\alpha_1\alpha_2) w_0^2 + 2\alpha_1\alpha_2 w_0w_1 } \right]}
                    {2\alpha_1\alpha_2 w_{-1}w_0} ,
\end{equation}
\begin{equation}
   g^{3}_{0} =  \frac{ 4\alpha_1\alpha_2 w_{-2} w_0 - 2\alpha_1\alpha_2 w_{-1}^2 - (\alpha_1(3+\alpha_2^2)+\alpha_2(1+3\alpha_1^2)) w_0 w_{-1}}
                     { 2\alpha_1\alpha_2 w_{-1} w_0} .
\end{equation}
\end{corollary}

The primary variables can be recovered from the the $ \{f^{j}_n, g^{j}_n \} $ variables using the formulae given in
the following corollary and is a direct consequence of the theory in \cite{FW_2006c}, \cite{Wit_2009}.
\begin{corollary}\label{Recover_ColCorrSqL}
The Toeplitz determinants \eqref{ssColumn} and \eqref{ToepCol}, abbreviated as $ \langle \sigma_{0,0}\sigma_{0,N} \rangle^{\square} \eqqcolon I_{N} $,
are computed using the generic recurrence relations of Cor. \ref{Recover}
with the initial conditions $ I_0=1 $, $ I_1=w_0 $. In turn the pair of reflection coefficients $ r_n, \bar{r}_n $ are computed firstly using
the recurrence relation
\begin{equation}
       \frac{r_{n+1}}{r_n} = \frac{\alpha_1\alpha_2 n}{n+1} 
       \frac{(1-\alpha_1^2)\left[ 1-\alpha_2^2 f^{2}_{n} \right] - (1-\alpha_2^2)\left[ f^{1}_{n}-\alpha_1^2f^{3}_{n} \right] }
            {\alpha_1(1-\alpha_1^2)\left[ 1-\alpha_2^4 f^{2}_{n} \right] - \alpha_2(1-\alpha_2^2)\left[ f^{1}_{n}-\alpha_1^4 f^{3}_{n} \right] }, \quad n \geq 1.
\end{equation} 
Secondly $ \bar{r}_n $ is computed through another coefficient $ \lambda_{n} $ which satisfies
\begin{multline}
   (n+1)\lambda_{n+1}-n\lambda_{n} = g^{3}_n - \frac{(\alpha_1+\alpha_2) (1+\alpha_1\alpha_2)}{\alpha_1 \alpha_2} n 
\\  
   + \frac{n+1}{\alpha_1 \alpha_2}
     \frac{\left[ \splitdfrac{ (1-\alpha_1^2)\left[ (\alpha_1+\alpha_2(1+\alpha_1^2))-(1+\alpha_1(\alpha_1+\alpha_2))\alpha_2^3f^{2}_{n} \right]}
                      {- (1-\alpha_2^2)\left[ (\alpha_2+\alpha_1(1+\alpha_2^2)) f^{1}_{n}-(1+\alpha_2(\alpha_1+\alpha_2))\alpha_1^3f^{3}_{n} \right]} \right]}
          {(1-\alpha_1^2)\left[ 1-\alpha_2^2 f^{2}_{n} \right] - (1-\alpha_2^2)\left[ f^{1}_{n}-\alpha_1^2 f^{3}_{n} \right] } ,
\label{}
\end{multline} 
and the generic relation $ \lambda_{n+1}-\lambda_{n} = r_{n+1}\bar{r}_{n} $.
Alternatively one can use
\begin{multline}
   (n+1)\bar{\lambda}_{n+1}-n\bar{\lambda}_{n} = g^{1}_n -\frac{(1+\alpha_1\alpha_2)}{2\alpha_1\alpha_2}\left[ (2n-1)\alpha_1+(2n+1)\alpha_2 \right]
\\
   + \frac{n+1}{\alpha_1 \alpha_2}
     \frac{\alpha_1(1-\alpha_1^2)\left[ 1-\alpha_2^4f^{2}_{n} \right] - \alpha_2 (1-\alpha_2^2)\left[ f^{1}_{n}-\alpha_1^4f^{3}_{n} \right] }
          {(1-\alpha_1^2)\left[ 1-\alpha_2^2 f^{2}_{n} \right] - (1-\alpha_2^2)\left[ f^{1}_{n}-\alpha_1^2 f^{3}_{n} \right] } ,
\label{}
\end{multline}
along with the previous generic relations.
\end{corollary}

\subsection{Diagonal Correlations of the Square Lattice Model}
The diagonal correlations of the rectangular Ising model also have a Toeplitz determinant form (see Eqs. (3.29,30,31) of \cite{MW_1973})
given by
\begin{equation} 
  \langle\sigma_{0,0}\sigma_{N,N}\rangle  = \det(w_{j-k})_{0\leq j,k\leq N-1}, 
\label{DCSL_Toeplitz} 
\end{equation} 
where the Fourier coefficients $ w_{n} $ 
\begin{equation}
   w_{n} = \int^{\pi}_{-\pi}\frac{d\theta}{2\pi}
           \frac{k\cos(n\theta)-\cos((n+1)\theta)}{\sqrt{\left(1-2k\cos(\theta)+k^2\right)}} ,
\end{equation} 
are defined by the weight
\begin{equation} 
  w(\zeta) = \left[ \frac{(1-k^{-1}\zeta^{-1})}{(1-k^{-1}\zeta)} \right]^{1/2} ,
\label{DCSLwgt}
\end{equation}
and where $ k = S\bar{S} $, $ S=\sinh 2K_1, \bar{S}=\sinh 2K_2 $. This arises from \eqref{Telement} as $ z_3 \to 0 $ or 
$ v_3\to 1 $  and thus $ a=4z_1z_2, b=0, c=(1-z_1^2)(1-z_2^2) $, so that $ k=a/c $. We will also use 
\begin{equation}
 \alpha = \frac{1}{k} = \frac{(1-z_1^2) (1-z_2^2)}{4 z_1 z_2} .
\end{equation} 
Under this limiting process the singularities behave like
\begin{gather}
  \zeta _1 = \frac{4}{\alpha}(v_3-1)^{-2}+{\rm O}((v_3-1)^{-1}) ,
\\
  \zeta_2 = \alpha+{\rm O}((v_3-1)^{1}) ,
\\
  \zeta_3 = \frac{1}{\alpha}+{\rm O}((v_3-1)^{1}) ,
\\
  \zeta _4  = \frac{\alpha}{4}(v_3-1)^2+{\rm O}((v_3-1)^{3}) ,
\end{gather}
and so two of these merge with those at $ 0, \infty $ leaving two at finite, non-zero locations.
Therefore the three variable Garnier system reduces to the single variable Painlev\'e VI system.
We first consider the limit of Corollary \ref{LinearRR}.

\begin{corollary}[\cite{AP_2002a}]
Under the limiting process $ z_3 \to 0 $ or $ v_3\to 1 $ Eq. \eqref{LRR_triangle} for the Toeplitz matrix elements $ w_n $ 
reduces to the second order linear difference equation
\begin{equation}
  \alpha(2n-3)w_{n-2} - 2\left[(\alpha^2+1)n-1\right]w_{n-1} + \alpha(2n+1)w_{n} = 0.
\end{equation}
This is identical to Eq. (2.18) of \cite{AP_2002a}.
\end{corollary}

Next we check the limit of the nonlinear system in Corollary \ref{NonlinearRR}.
\begin{proposition}\label{SLdiagonalCorr}
Through the limiting process $ z_3 \to 0 $ or $ v_3\to 1 $ the dynamical variables \eqref{ReccVars} behave as
\begin{gather}
 f^{1}_{n} = (v_3-1)^2F^{1}_{n} ,\quad f^{2}_{n} = F^{2}_{n} ,\quad f^{3}_{n} = (v_3-1)^{-4}F^{3}_{n} ,
\\
 g^{1}_{n} = (v_3-1)^{-2}G^{1}_{n} ,\quad g^{2}_{n} = (v_3-1)^{-2}G^{2}_{n} ,\quad g^{3}_{n} = (v_3-1)^{-2}G^{3}_{n} ,
\label{ScaleToDSL}
\end{gather} 
where $ F^{j}_{n}, G^{j}_{n} = {\rm O}(1) $ in this limiting process. Let us define the change of variables
\begin{equation}
  F^{2}_{n} =\alpha^{-4}f_{n} ,\quad G^{2}_{n} = \frac{4}{\alpha}g_{n} .
\end{equation} 
Under the conditions $g_{n}+n\alpha^{-1} \neq 0 $ and $ g_{n}+(n+\frac{1}{2})\alpha^{-1}-\frac{1}{2}\alpha \neq 0 $
the dynamical equations \eqref{1stLP:1}-\eqref{2ndLP:3} reduce under the limit to the pair of coupled first order difference equations
\begin{gather}
   \alpha^{-2}f_{n}f_{n+1} = \frac{\left[ g_{n}+(n+1)\alpha-\alpha^{-1} \right]\left[ g_{n}+(n+\frac{1}{2})\alpha-\frac{1}{2}\alpha^{-1} \right]}
                                  {\left[ g_{n}+n\alpha^{-1} \right]\left[ g_{n}+(n+\frac{1}{2})\alpha^{-1}-\frac{1}{2}\alpha \right]} ,
\label{1stLaxPairDIAGSL}\\
   g_{n}+g_{n-1}+2n\alpha^{-1}-\tfrac{1}{2}(\alpha+\alpha^{-1}) + (n+\tfrac{1}{2})\frac{(\alpha^2-1)}{\alpha}\frac{1}{1-f_{n}} + (n+\tfrac{1}{2})\frac{\alpha(\alpha^2-1)}{\alpha^2-f_{n}} = 0 .
\label{2ndLaxPairDIAGSL}
\end{gather}
This is precisely the $ D^{(1)}_4 $ discrete Painlev\'e system in the Sakai scheme otherwise known as 
``discrete Painlev\'e V''. These are subject to the initial conditions
\begin{equation}
    f_0 = \alpha\frac{w_{-1}+\alpha w_0}{\alpha w_{-1}+w_0} ,
\quad   
    g_0 = \frac{1}{2}\left( \frac{w_{-1}}{w_0}-\frac{w_0}{w_{-1}} \right) .
\end{equation} 
\end{proposition}
\begin{proof}
We note some preliminary expansions of a number of auxiliary quantities appearing in the Lax pairs \eqref{1stLP:1}-\eqref{1stLP:3}
and \eqref{2ndLP:1}-\eqref{2ndLP:3} 
\begin{align}
   \frac{2\Delta v_3}{v_1^2v_2^2(1+v_3)^6}\left[ \frac{2(\Delta^2-4 v_1^2 v_2^2 v_3(1+v_3)^2)}{\Gamma+\Delta}-\Delta \right]
   & = -\tfrac{1}{2}+{\rm O}((v_3-1)^2) ,
\\
   -\frac{2\Delta v_3}{v_1^2 v_2^2 (1+v_3)^6} \left[ \frac{2\left(\Delta ^2-4 v_1^2 v_2^2 v_3(1+v_3)^2\right)}{\Gamma-\Delta}+\Delta \right]
   & = -\frac{2(1-\alpha^2)}{\alpha^4}(v_3-1)^{-2}+{\rm O}((v_3-1)^{-1}) ,
\\
   -\frac{2\Delta v_3}{v_1^2v_2^2(1-v_3)^6}\left[ \frac{2\left(\Delta ^2+4 v_1^2 v_2^2 v_3(1-v_3)^2 \right)}{\Gamma-\Delta}+\Delta \right]
   & = -\frac{128}{\alpha^4}(v_3-1)^{-8}+{\rm O}((v_3-1)^{-7}) ,
\\
   \frac{2\Delta v_3}{v_1^2 v_2^2(1-v_3)^6} \left[ \frac{2\left(\Delta^2+4 v_1^2 v_2^2 v_3(1-v_3)^2 \right)}{\Gamma+\Delta}-\Delta \right]
   & = -2(1-\alpha^2)(v_3-1)^{-2}+{\rm O}((v_3-1)^{-1}) ,
\end{align} 
along with
\begin{align}
   \mathcal{R}_1 & = \frac{\alpha}{4}g^{1}_{n}(v_3-1)^2 + {\rm O}((v_3-1)^{3}) ,
\\
   \mathcal{R}_2 & = \left[ \alpha^{-1} g^{1}_{n}+\alpha^{-2} g^{2}_{n}+\alpha^{-3} g^{3}_{n}+\alpha^{-4} \right] + {\rm O}((v_3-1)) ,
\\
   \mathcal{R}_3 & = \frac{256}{\alpha^4}(v_3-1)^{-8} + {\rm O}((v_3-1)^{-7}) ,
\\
   \mathcal{R}_4 & = \left[ \alpha g^{1}_{n}+\alpha^2 g^{2}_{n}+\alpha^3 g^{3}_{n}+\alpha^4 \right] + {\rm O}((v_3-1)) ,
\end{align}
and
\begin{align}
  \mathcal{S}_1 & = 2(1-v_1^2)^2(1-v_2^2)^2 \left[ 1-(1-\alpha^2) f^{1}_{n}-\alpha^2 f^{2}_{n} \right](v_3-1)^2 + {\rm O}((v_3-1)^3) ,
\\
  \mathcal{S}_2 & = -4(1-v_1^2)^3(1-v_2^2)^3 (1-\alpha^2)f^{1}_{n} (v_3-1)^2 + {\rm O}((v_3-1)^3) ,
\\
  \mathcal{S}_3 & = 16(1-v_1^2)^3(1-v_2^2)^3 \left[ 1-(1-\alpha^2) f^{1}_{n}-\alpha^2 f^{2}_{n} \right](v_3-1)^2 + {\rm O}((v_3-1)^4) ,
\\
  \mathcal{S}_4 & = 8(1-v_1^2)^4(1-v_2^2)^4 \left[ 1-(1-\alpha^4) f^{1}_{n}-\alpha^4 f^{2}_{n} \right](v_3-1)^4 + {\rm O}((v_3-1)^5) .
\end{align}
Using the scalings \eqref{ScaleToDSL} in \eqref{1stLP:1}, and assuming  
$ G^{1}_{n}+\alpha G^{2}_{n}+\alpha^2 G^{3}_{n} \neq 0 $ and $ \alpha G^{1}_{n}+\alpha^2 G^{2}_{n}+\alpha^3 G^{3}_{n}-2(1-\alpha^2) \neq 0 $
we deduce from the leading order in this expansion that
\begin{equation}
 (\alpha G^{1}_{n}-4n)(\alpha G^{1}_{n}-4n-2) = 0 .
\label{expand1st}
\end{equation} 
Treating \eqref{1stLP:3} in a similar manner we conclude, again under the above conditions, that
\begin{equation}
   (\tfrac{1}{4}\alpha G^{3}_{n}+1)(\tfrac{1}{2}\alpha G^{3}_{n}+1) = 0 .
\label{expand2nd}
\end{equation} 
Which of these two sets of solutions are relevant can be settled by the initial conditions, which we undertake subsequently.
The other member of the first set of the Lax pair, \eqref{1stLP:2}, now has at leading order the expression
\begin{equation}
  \frac{1}{\alpha^8} \left\{ 
   \alpha^6 F^{2}_{n} F^{2}_{n+1}-
                      \frac{\left[ \alpha^2 G^{1}_{n}+\alpha  G^{2}_{n}+G^{3}_{n} \right]}
                           {\left[ G^{1}_{n}+\alpha  G^{2}_{n}+\alpha^2 G^{3}_{n} \right]}
                      \frac{\left[ \alpha^3 G^{1}_{n}+\alpha^2 G^{2}_{n}+\alpha  G^{3}_{n}-2(1-\alpha^2) \right]}
                           {\left[ \alpha  G^{1}_{n}+\alpha^2 G^{2}_{n}+\alpha^3 G^{3}_{n}-2(1-\alpha^2) \right]}
                      \right\} .
\label{aux1stLP}
\end{equation}
Undertaking the expansion of the initial values given in \eqref{initf1}-\eqref{initf3} and \eqref{initg1}-\eqref{initg3}
we compute the leading terms as
\begin{align}
   f^{1}_0 & = \frac{w_0}{4(\alpha w_{-1}+w_0)}(v_3-1)^2 + {\rm O}((v_3-1)^3) ,
\\
   f^{2}_0 & = \frac{w_{-1}+\alpha w_0}{\alpha^3(\alpha w_{-1}+w_0)} + {\rm O}(v_3-1) ,
\\
   f^{3}_0 & = -\frac{16w_{-1}}{\alpha^3(\alpha w_{-1}+w_0)}(v_3-1)^{-4} + {\rm O}((v_3-1)^{-3}) ,
\\
   g^{1}_0 & = \frac{2}{\alpha}(v_3-1)^{-2} + {\rm O}((v_3-1)^{-1}) ,
\\
   g^{2}_0 & = -\frac{2\alpha w_{-1}w_0+w_0^2-3w_{-1}w_1}{\frac{1}{2}\alpha w_{-1}w_0}(v_3-1)^{-2} + {\rm O}((v_3-1)^{-1}) ,
\\
   g^{3}_0 & = -\frac{2}{\alpha}(v_3-1)^{-2} + {\rm O}((v_3-1)^{-1}) ,
\end{align}
and thus
\begin{equation}
   F^{1}_0 = \frac{w_0}{4(\alpha w_{-1}+w_0)} ,
\quad   
   F^{2}_0 = \frac{w_{-1}+\alpha w_0}{\alpha^3(\alpha w_{-1}+w_0)} ,
\quad  
   F^{3}_0 = -\frac{16w_{-1}}{\alpha^3(\alpha w_{-1}+w_0)} ,
\end{equation}
\begin{equation}   
   G^{1}_0 = \frac{2}{\alpha} ,
\quad   
   G^{2}_0 = \frac{2}{\alpha}\left( \frac{w_{-1}}{w_0}-\frac{w_0}{w_{-1}} \right) ,
\quad    
   G^{3}_0 = -\frac{2}{\alpha} .
\end{equation}
From these initial values we see that the choice for the solutions to \eqref{expand1st} and \eqref{expand2nd} are
\begin{equation}
  G^{1}_{n} = \frac{4n+2}{\alpha} ,\quad G^{3}_{n} = -\frac{2}{\alpha} .
\label{aux1st}
\end{equation} 
Now we turn our attention to the second set of the Lax pair and in particular expanding \eqref{2ndLP:1} we find at the 
leading order that
\begin{equation}
   G^{1}_{n}+G^{1}_{n-1}-\frac{2(2n-1)}{\alpha}-\frac{4n}{\alpha}
   + \frac{16n}{\alpha}\frac{(1-\alpha^2)F^{1}_{n}}{4(1-\alpha^2)F^{1}_{n}+\alpha^4 F^{2}_{n}-1} .
\end{equation} 
Equating this to zero and substituting the first of the solutions \eqref{aux1st} for $ G^{1}_{n} $ we compute $ F^{1}_{n} $ as
\begin{equation}
   F^{1}_{n} = \frac{\alpha^4 F^{2}_{n}-1}{4(\alpha^2-1)(2n+1)} .
\label{aux2nd:1}
\end{equation} 
Continuing we expand \eqref{2ndLP:3} to leading order and deduce the expression
\begin{equation}
  G^{3}_{n}+ G^{3}_{n-1}-\frac{4n-2}{\alpha}+\frac{4}{\alpha}(n+1)\frac{\alpha^2 F^{2}_{n}-1}{\alpha^2 F^{2}_{n}-\frac{1}{16}\alpha^2(1-\alpha^2)F^{3}_{n}-1} .
\end{equation} 
Utilising the second solution of \eqref{aux1st} we solve for $ F^{3}_{n} $ as
\begin{equation}
   F^{3}_{n} = \frac{16(\alpha^2 F^{2}_{n}-1)}{\alpha^2(\alpha^2-1)(2n+1)} .
\label{aux2nd:2}
\end{equation}
Lastly we expand \eqref{2ndLP:2} and compute to leading order
\begin{multline}
  G^{2}_{n}+G^{2}_{n-1}+\frac{4(1+\alpha^2)(n-1)}{\alpha^2}+4+\frac{4n}{\alpha^2}
\\
  +\frac{4(n+1)}{\alpha^2} \frac{1-\alpha^4 F^{2}_{n}}{\alpha^2 F^{2}_{n}-\frac{1}{16}\alpha^2(1-\alpha^2) F^{3}_{n}-1}
  -\frac{4(1-\alpha^2)n}{\alpha^2} \frac{4 F^{1}_{n}-1}{4(1-\alpha^2) F^{1}_{n}+\alpha^4 F^{2}_{n}-1} = 0 .
\label{aux2ndLP}
\end{multline}
Finally employing the solutions \eqref{aux1st} and the relations \eqref{aux2nd:1} and \eqref{aux2nd:2} in the two remaining
equalities, \eqref{aux1stLP} and \eqref{aux2ndLP}, we arrive at \eqref{1stLaxPairDIAGSL} and \eqref{2ndLaxPairDIAGSL}
respectively.
\end{proof}

\begin{proposition}\label{recoverSLdiagonalCorr}
The diagonal correlations of the square lattice Ising model are governed by the generic relations of Corollary \ref{Recover}
along with
\begin{equation}
   \frac{r_{n+1}}{r_{n}} = \frac{\alpha-\alpha^{-1}f_{n}}{1-f_{n}} ,
\label{recoverDIAGSL:1}
\end{equation} 
and
\begin{equation}
   (n+\tfrac{1}{2})\lambda_{n+1}-(n-\tfrac{1}{2})\lambda_{n} = -n(\alpha+\alpha^{-1}) - g_{n} + (n+\tfrac{1}{2})\frac{1-f_{n}}{\alpha-\alpha^{-1}f_{n}} .
\label{recoverDIAGSL:2}
\end{equation} 
\end{proposition}
\begin{proof}
Expanding the relation \eqref{recover:1} as $ v_3 \to 1 $ and using the results of the previous proposition we find the 
leading order gives 
\begin{equation}
    \frac{r_{n+1}}{r_{n}}-\alpha\frac{n}{n+1}\frac{\alpha^2 F^{2}_{n}-1-\frac{1}{16}\alpha^2(1-\alpha^2) F^{3}_{n}}{4(1-\alpha^2) F^{1}_{n}+\alpha^4 F^{2}_{n}-1} .
\end{equation}
Employing \eqref{aux2nd:1} and \eqref{aux2nd:2} we deduce \eqref{recoverDIAGSL:1}.
\end{proof}

\begin{remark}
The Lax pair \eqref{1stLaxPairDIAGSL} and \eqref{2ndLaxPairDIAGSL} can also be verified directly from the theory in \S 4 of
\cite{Wit_2009}. From the weight \eqref{DCSLwgt} the configuration of singularities is given by
\begin{equation}
\left\{
\begin{array}{cccc}
   0 & \alpha & \alpha^{-1} & \infty \\
  -1/2 &  1/2 & -1/2 &  1/2
\end{array}
\right\} , 
\end{equation} 
from which we compute the spectral data as $ e_0=1, e_1=\alpha+\alpha^{-1}, e_2=1 $ and $ m_0=-1/2, m_1=-\alpha, m_2=-1/2 $.
Employing the parameterisation of the spectral polynomials as
\begin{gather}
  \frac{\kappa_{n+1}}{\kappa_{n}}\Theta_n(\zeta) = (n+\tfrac{1}{2})\zeta+\vartheta_n ,
\\
  \Omega_{n}(\zeta) = \tfrac{3}{4}\zeta^2-(g_n+\tfrac{1}{2}\alpha)\zeta-n-\tfrac{1}{4} ,
\end{gather}
and making the definition 
\begin{equation}
  \alpha^{-2}f_n = \frac{\Theta_n(\alpha^{-1})}{\Theta_n(\alpha)} ,
\end{equation}
we can use Eq. (2.18) and Eq. (1.33) of \cite{Wit_2009} to deduce \eqref{1stLaxPairDIAGSL} and \eqref{2ndLaxPairDIAGSL}
respectively. This Lax pair differ from those of Prop. 4.1 of \cite{Wit_2009} only in the rescaling of the independent variable 
$ \zeta $. Alternative, but entirely equivalent, recurrence relations have been given in Corollary 5.3 of \cite{FW_2005}
and Corollary 1 of \cite{Wit_2007}.
\end{remark}

\begin{remark}
The forgoing {\it generic} theory does not apply without modifications to the critical cases of the Curie, N\'eel and
disordered points where $ \Delta=0, \bar{\Delta}=0 $ because of the undefined nature of the simultaneous vanishing of numerators
and denominators. There are two strategies one could adopt here: perform a local expansion about the critical points or
adopt the more direct approach of evaluation of the primary variables from first principles. However we defer this task to a
subsequent study because of the lengthy technical issues involved.
\end{remark}

\begin{remark}
We have not attempted to discuss the hypergeometric function evaluation of the Toeplitz matrix elements $ w_n $ in the general case of
the various regimes: $ T < T_C, T_N $, or $ T_N < T < T_D $ and $ T > T_D $. Suffice it to say that one would expect the three-variable
Lauricella $ D $ hypergeometric functions would appear in the diagonal correlations on the generic anisotropic triangular lattice
whereas the third complete elliptic integral would describe the row/column correlations of the anisotropic square lattice, in addition
to the first and second complete elliptic integrals.
\end{remark}

\section*{Acknowledgements}
This research has been supported by the Australian Research Council Centre of Excellence, Mathematics and Statistics of Complex
Systems. We also gratefully acknowledge support from the Simons Center for Geometry and Physics, Stony Brook University where some
of the research for this paper was performed. Finally we are indebted to Barry McCoy and Jacques Perk for their keen interest in 
developments on the fundamental understanding of the Ising model and for their constant encouragement.   

\setcounter{equation}{0}
\bibliographystyle{amsplain}
\inputencoding{cp1252}
\bibliography{moment,nonlinear,random_matrices}

\def\cprime{$'$} \def\cprime{$'$} \def\cprime{$'$}
\providecommand{\bysame}{\leavevmode\hbox to3em{\hrulefill}\thinspace}
\providecommand{\MR}{\relax\ifhmode\unskip\space\fi MR }
\providecommand{\MRhref}[2]{%
  \href{http://www.ams.org/mathscinet-getitem?mr=#1}{#2}
}
\providecommand{\href}[2]{#2}
\begin{thebibliography}{10}

\bibitem{AP_2002a}
H.~Au-Yang and J.~H.~H. Perk, \emph{Correlation functions and susceptibility in
  the {$Z$}-invariant {I}sing model}, MathPhys odyssey, 2001, Prog. Math.
  Phys., vol.~23, Birkh\"auser Boston, Boston, MA, 2002, pp.~23--48.
  \MR{MR1903971 (2003i:82012)}

\bibitem{BM_2006}
N.~D. Blakeley and R.~P. Millane, \emph{Exact enumeration of the ground states
  of a triangular {I}sing antiferromagnet}, Comput. Phys. Comm. \textbf{174}
  (2006), no.~3, 198--201. \MR{MR2197617 (2006h:82011)}

\bibitem{BD_2002}
A.~Borodin and P.~Deift, \emph{Fredholm determinants, {J}imbo-{M}iwa-{U}eno
  {$\tau$}-functions, and representation theory}, Comm. Pure Appl. Math.
  \textbf{55} (2002), no.~9, 1160--1230. \MR{1 908 746}

\bibitem{Bug_1996}
E.~A. Bugri{\u\i}, \emph{Solution of the {$2$}{D} {I}sing model on a triangular
  lattice by the auxiliary {$q$}-deformed {G}rassmann field method}, Theoret.
  Math. Phys. \textbf{109} (1996), no.~3, 1590--1607 (English). \MR{MR1472481
  (98h:82012)}

\bibitem{CGNP_2011}
Y.~Chan, A.J. Guttmann, B.G. Nickel, and J.H.H. Perk, \emph{The {I}sing
  susceptibility scaling function}, J. Statist. Phys. \textbf{145} (2011),
  no.~3, 549--590.

\bibitem{FIK_1991}
A.~S. Fokas, A.~R. It{\cydot{s}}, and A.~V. Kitaev, \emph{Discrete {P}ainlev\'e
  equations and their appearance in quantum gravity}, Comm. Math. Phys.
  \textbf{142} (1991), no.~2, 313--344. \MR{93a:58080}

\bibitem{FW_2004a}
P.~J. Forrester and N.~S. Witte, \emph{Application of the tau-function theory
  of {P}ainlev\'e equations to random matrices: {P-VI}, the {JUE}, {CyUE},
  {cJUE} and scaled limits}, Nagoya Math. J. \textbf{174} (2004), 29--114.

\bibitem{FW_2005}
\bysame, \emph{Discrete {P}ainlev\'e equations for a class of {$P\sb {\rm VI}$}
  {$\tau$}-functions given as {${\rm U}(N)$} averages}, Nonlinearity
  \textbf{18} (2005), no.~5, 2061--2088. \MR{MR2164732}

\bibitem{FW_2006c}
\bysame, \emph{Bi-orthogonal polynomials on the unit circle, regular
  semi-classical weights and integrable systems}, Constr. Approx. \textbf{24}
  (2006), no.~2, 201--237. \MR{MR2239121 (2007k:41003)}

\bibitem{GORS_1998}
B.~Grammaticos, Y.~Ohta, A.~Ramani, and H.~Sakai, \emph{Degeneration through
  coalescence of the {$q$}-{P}ainlev\'e {VI} equation}, J. Phys. A \textbf{31}
  (1998), no.~15, 3545--3558. \MR{MR1626199 (99g:39007)}

\bibitem{HTM_1992}
T.~Horiguchi, K.~Tanaka, and T.~Morita, \emph{Low-temperature behavior of
  antiferromagnetic {I}sing model on triangular lattice}, J. Phys. Soc. Japan
  \textbf{61} (1992), no.~1, 64--69. \MR{MR1158223 (92m:82023)}

\bibitem{HS_1950}
K.~Husimi and I.~Sy{\^o}zi, \emph{The statistics of honeycomb and triangular
  lattice. {I}}, Progress Theoret. Physics \textbf{5} (1950), 177--186.
  \MR{0039628 (12,576f)}

\bibitem{IP_1995}
T.~A. Ivanova and A.~D. Popov, \emph{Some new integrable equations from the
  self-dual {Y}ang-{M}ills equations}, Phys. Lett. A \textbf{205} (1995),
  no.~2-3, 158--166. \MR{96j:58080}

\bibitem{IKSY_1991}
K.~Iwasaki, H.~Kimura, S.~Shimomura, and M.~Yoshida, \emph{From {G}auss to
  {P}ainlev\'e}, Aspects of Mathematics, E16, Friedr. Vieweg \& Sohn,
  Braunschweig, 1991, A modern theory of special functions. \MR{MR1118604
  (92j:33001)}

\bibitem{JM_1980}
M.~Jimbo and T.~Miwa, \emph{Deformation of linear ordinary differential
  equations. {I}}, Proc. Japan Acad. Ser. A Math. Sci. \textbf{56} (1980),
  no.~4, 143--148. \MR{MR575994 (84k:14009)}

\bibitem{KO_1984}
H.~Kimura and K.~Okamoto, \emph{On the polynomial {H}amiltonian structure of
  the {G}arnier systems}, J. Math. Pures Appl. (9) \textbf{63} (1984), no.~1,
  129--146. \MR{86c:34008}

\bibitem{Kos_1989}
N.~A. Kostov, \emph{Quasi-periodic solutions of the integrable dynamical
  systems related to {H}ill's equation}, Lett. Math. Phys. \textbf{17} (1989),
  no.~2, 95--108. \MR{91a:58080}

\bibitem{LYC_2008}
Y.~L. Loh, D.-X. Yao, and E.~W. Carlson, \emph{Thermodynamics of {I}sing spins
  on the triangular {K}agome lattice: Exact analytical method and {M}onte
  {C}arlo simulations}, Physical Review B \textbf{77} (2008), no.~13, 134402.

\bibitem{Mag_1995a}
A.~P. Magnus, \emph{Painlev\'e-type differential equations for the recurrence
  coefficients of semi-classical orthogonal polynomials}, J. Comput. Appl.
  Math. \textbf{57} (1995), no.~1-2, 215--237. \MR{96f:42027}

\bibitem{Mal_1922}
J.~Malmquist, \emph{Sur les {\'e}quations diff{\'e}rentielles du second ordre
  dont l'int{\'e}grale g{\'e}n{\'e}rale a ses points critiques fixes}, Arkiv
  Mat., Astron. Fys. \textbf{18} (1922), no.~8, 1--89.

\bibitem{MW_1973}
B.~McCoy and T.~T. Wu, \emph{The {T}wo-{D}imensional {I}sing {M}odel}, Harvard
  University Press, Harvard, 1973.

\bibitem{MS_2000}
P.~Menotti and D.~Seminara, \emph{A{D}{M} approach to $2+1$ dimensional gravity
  coupled to particles}, Ann. Physics \textbf{279} (2000), no.~2, 282--310.
  \MR{1 743 926}

\bibitem{MGP_2003}
R.~P. Millane, A.~Goyal, and R.~C. Penney, \emph{Ground states of the
  antiferromagnetic {I}sing model on finite triangular lattices of simple
  shape}, Phys. Lett. A \textbf{311} (2003), no.~4-5, 347--352. \MR{1980417
  (2004g:82020)}

\bibitem{MPW_1963}
E.~W. Montroll, R.~B. Potts, and J.~C. Ward, \emph{Correlations and spontaneous
  magnetization of the two-dimensional {I}sing model}, J. Math. Phys.
  \textbf{4} (1963), no.~2, 308--322. \MR{0148406 (26 \#5913)}

\bibitem{New_1950a}
G.~F. Newell, \emph{Crystal statistics of a two-dimensional {I}sing lattice},
  Phys. Rev. \textbf{78} (1950), 444--449.

\bibitem{New_1950}
\bysame, \emph{Crystal statistics of a two-dimensional triangular {I}sing
  lattice}, Physical Rev. (2) \textbf{79} (1950), 876--882. \MR{0039631
  (12,576i)}

\bibitem{NRGO_2001}
F.~W. Nijhoff, A.~Ramani, B.~Grammaticos, and Y.~Ohta, \emph{On discrete
  {P}ainlev\'e equations associated with the lattice {K}d{V} systems and the
  {P}ainlev\'e {VI} equation}, Stud. Appl. Math. \textbf{106} (2001), no.~3,
  261--314. \MR{1819383 (2003d:39028)}

\bibitem{Oka_1987}
K.~Okamoto, \emph{Studies on the {P}ainlev\'e equations. {I}. {S}ixth
  {P}ainlev\'e equation ${P}\sb {{\rm {V}{I}}}$}, Ann. Mat. Pura Appl. (4)
  \textbf{146} (1987), 337--381. \MR{88m:58062}

\bibitem{Ple_1988}
V.~N. Plechko, \emph{Grassmann path-integral solution for a class of triangular
  type decorated {I}sing models}, Phys. A \textbf{152} (1988), no.~1-2, 51--97.
  \MR{967719 (90a:82077)}

\bibitem{Pot_1955}
R.~B. Potts, \emph{Combinatorial solution of the triangular {I}sing lattice},
  Proc. Phys. Soc. Sect. A. \textbf{68} (1955), 145--148. \MR{0070532
  (16,1190g)}

\bibitem{Sak_2001}
H.~Sakai, \emph{Rational surfaces associated with affine root systems and
  geometry of the {P}ainlev\'e equations}, Comm. Math. Phys. \textbf{220}
  (2001), no.~1, 165--229. \MR{1 882 403}

\bibitem{Ste_1964}
J.~Stephenson, \emph{{I}sing-model spin correlations on the triangular
  lattice}, J. Math. Phys. \textbf{5} (1964), 1009--1024. \MR{MR0168310 (29
  \#5574)}

\bibitem{Ste_1966}
\bysame, \emph{{I}sing model spin correlations on the triangular lattice. {II}.
  {F}ourth-order correlations}, J. Math. Phys. \textbf{7} (1966), no.~6,
  1123--1132.

\bibitem{Ste_1970a}
\bysame, \emph{{I}sing-model spin correlations on the triangular lattice.
  {III}. {I}sotropic antiferromagnetic lattice}, J. Math. Phys. \textbf{11}
  (1970), no.~2, 413--419.

\bibitem{Ste_1970b}
\bysame, \emph{{I}sing-model spin correlations on the triangular lattice. {IV}.
  {A}nisotropic ferromagnetic and antiferromagnetic lattices}, J. Math. Phys.
  \textbf{11} (1970), no.~2, 420--431.

\bibitem{Syo_1950}
I.~Sy{\^o}zi, \emph{The statistics of honeycomb and triangular lattice. {II}},
  Progress Theoret. Physics \textbf{5} (1950), 341--351. \MR{0039629 (12,576g)}

\bibitem{Tem_1950}
H.~N.~V. Temperley, \emph{Statistical mechanics of the two-dimensional
  assembly}, Proc. Roy. Soc. London. Ser. A. \textbf{202} (1950), 202--207.
  \MR{0039630 (12,576h)}

\bibitem{Ton_1994}
G.~Tondo, \emph{A connection between the {H}\'enon-{H}eiles system and the
  {G}arnier system}, Teoret. Mat. Fiz. \textbf{99} (1994), no.~3, 552--559.
  \MR{95i:58096}

\bibitem{Wan_1950}
G.~H. Wannier, \emph{Antiferromagnetism. {T}he triangular {I}sing net},
  Physical Rev. (2) \textbf{79} (1950), 357--364. \MR{0039627 (12,576e)}

\bibitem{Wit_2007}
N.~S. Witte, \emph{Isomonodromic deformation theory and the next-to-diagonal
  correlations of the anisotropic square lattice {I}sing model}, J. Phys. A:
  Math. Theor. \textbf{40} (2007), no.~24, F491--F501.

\bibitem{Wit_2009}
N.~S. Witte, \emph{Bi-orthogonal systems on the unit circle, regular
  semi-classical weights and the discrete {G}arnier equations}, Int. Math. Res.
  Not. \textbf{2009,} (2009), no.~6, 988--1025. \MR{MR2487490}

\bibitem{WM_2009}
D.~H. Wojtas and R.~P. Millane, \emph{Two-point correlation function for the
  triangular {I}sing antiferromagnet}, Phys. Rev. E \textbf{79} (2009), no.~4,
  041123.

\bibitem{WMTB_1976}
T.~T. Wu, B.~M. McCoy, C.~A. Tracy, and E.~Barouch, \emph{Spin-spin correlation
  functions for the two-dimensional {I}sing model: {E}xact theory in the
  scaling region}, Phys. Rev. B \textbf{13} (1976), no.~1, 316--374.

\bibitem{YLCM_2008}
D.-X. Yao, Y.~L. Loh, E.~W. Carlson, and M.~Ma, \emph{{X X Z} and {I}sing spins
  on the triangular {K}agome lattice}, Phys. Rev. B \textbf{78} (2008), no.~2,
  024428.

\bibitem{ZS_2005}
J.~Zheng and G.~Sun, \emph{Exact results for {I}sing models on the triangular
  {K}agom{\'e} lattice}, Phys. Rev. B \textbf{71} (2005), no.~5, 052408.

\end{thebibliography}
\inputencoding{utf8}

\end{document}